\documentclass{amsart}

\usepackage[top=1in,bottom=1in,left=1.1in,right=1.1in]{geometry}
\usepackage{amsmath,amssymb,amsfonts,amsthm,mathtools}
\usepackage{mathrsfs}
\usepackage{graphicx,adjustbox}
\usepackage{tikz-cd}
\usepackage{enumitem}
\usepackage{xcolor}
\usepackage{hyperref}

\hypersetup{
  colorlinks=true,
  linkcolor=blue,
  citecolor=blue,
  urlcolor=blue
}

\newtheorem{theorem}{Theorem}[section]
\newtheorem{proposition}[theorem]{Proposition}
\newtheorem{lemma}[theorem]{Lemma}
\newtheorem{corollary}[theorem]{Corollary}

\theoremstyle{definition}
\newtheorem{definition}[theorem]{Definition}
\newtheorem{construction}[theorem]{Construction}
\newtheorem{notation}[theorem]{Notation}

\theoremstyle{remark}
\newtheorem{remark}[theorem]{Remark}

\newcommand{\I}{\mathbf{I}}
\newcommand{\Mod}{\mathbf{Mod}}
\newcommand{\LMod}{\mathbf{Lmod}}
\newcommand{\Bimod}{\mathbf{Bimod}}
\newcommand{\Comm}{\mathbf{Comm}}
\newcommand{\Z}{\mathscr{Z}}
\newcommand{\Zl}{\mathrm{Z}^{l}_{1}}
\newcommand{\Zr}{\mathrm{Z}^{r}_{1}}
\newcommand{\boxt}{\,\Box\,}
\newcommand{\boxtsub}[1]{\,\Box_{#1}\,}

\newcommand{\Hom}{\mathbf{Hom}}
\newcommand{\stringdiagramfigurescale}{0.62}
\newcommand{\stringdiagramfigure}[1]{\adjustbox{valign=c}{\includegraphics[scale=\stringdiagramfigurescale]{#1}}}

\title[Centers of Algebras in Monoidal 2-Categories]{Centers of Algebras in Monoidal 2-Categories}
\author{Hao Xu}
\date{\today}

\begin{document}

\begin{abstract}
We introduce the left, right, and full center of an algebra in a monoidal 2-category and prove their Morita invariance. This categorifies Davydov's theory of centers of algebras in monoidal categories, and specializes to give a uniform, structural account of the Drinfeld center of a fusion category, the crossed braided Drinfeld center of a fusion category graded by a finite group, and the center of a central module monoidal category. Along the way, we develop a theory of 2-adjunctions between monoidal 2-categories and a base-change construction for module 2-categories.
\end{abstract}

\maketitle
\tableofcontents

\section{Introduction}

The purpose of this paper is to develop a 2-categorical version of Davydov's theory of centers and full centers of algebras in monoidal categories \cite{Dav10}. In the 1-categorical setting, an algebra in a braided monoidal category has left and right centers, while the full center can be defined for an algebra in any monoidal category as a commutative algebra in the Drinfeld center of the ambient monoidal category. One of the main structural points of the theory is that the full center is Morita invariant \cite{Dav10,DavHMod}.

We study the analogous story for algebras in monoidal 2-categories. If $\mathfrak{B}$ is a braided monoidal 2-category and $B$ is an algebra in $\mathfrak{B}$, we define the left center $\Zl(B)$ and right center $\Zr(B)$ by universal properties modeled on \emph{commuting objects} over $B$ (Definitions \ref{def:commuting_objects} and \ref{def:commuting_objects_morphisms}). The left center is terminal among objects mapping to $B$ that commute with $B$ using the braiding of $\mathfrak{B}$; the right center is obtained by replacing the braiding by the reverse braiding. Meanwhile, the full center of an algebra $A$ in a monoidal 2-category $\mathfrak{C}$ is defined as the terminal \emph{commuting half-braiding} of $A$ (Definitions \ref{def:commuting_half-braidings} and \ref{def:commuting_half-braidings_morphisms}), which turns out to be a braided algebra in the Drinfeld center $\Z_1(\mathfrak{C})$.

\subsection{Main results}

After recalling the necessary background on the coherence of tricategories \cite{GPS,Gur2} and on algebras, modules, and relative tensor products (§\ref{sec:preliminaries}), the paper establishes the following.

\begin{enumerate}[label=(\arabic*)]
    \item In §\ref{sec:monoidal_2-adjunction}, we establish the general theory of 2-adjunctions between monoidal 2-categories. In particular, we show that the right adjoint of an oplax monoidal 2-functor is lax monoidal, and vice versa, the left adjoint of a lax monoidal 2-functor is oplax monoidal (Construction \ref{cstr:lax_monoidal_structure_on_the_right_adjoint}). This in particular applies to strongly monoidal 2-functors, which characterizes the notion of 2-adjunction internal to the 3-category of monoidal 2-categories, (op)lax monoidal 2-functors, monoidal 2-natural transformations, and monoidal modifications (Corollary \ref{cor:monoidal_2-adjunction}).
    
    \item For an algebra $B$ in a braided monoidal 2-category $\mathfrak{B}$, the left and right centers $\Zl(B)$ and $\Zr(B)$, defined by their universal properties, carry canonical braided algebra structures (Proposition \ref{prop:left_center_braided}).
    
    \item The left and right centers are Morita invariant: if $A$ and $B$ are Morita equivalent algebras in $\mathfrak{B}$, then $\Zl(A)$ and $\Zl(B)$ (resp. $\Zr(A)$ and $\Zr(B)$) are equivalent as braided algebras in $\mathfrak{B}$ (Corollary \ref{cor:left_center_is_Morita_invariant}).
    
    \item If $A$ is an algebra in a braided monoidal 2-category $\mathfrak{B}$, the module 2-category $\Mod_{\mathfrak{B}}(A)$ satisfies a base-change principle: algebras in $\Mod_{\mathfrak{B}}(A)$ are algebras $B$ in $\mathfrak{B}$ equipped with a braided algebra map $A \to \Zl(B)$, and moreover, their module 2-categories are equivalent:
    \[\Mod_{\mathfrak{B}}(B) \simeq \Mod_{\Mod_{\mathfrak{B}}(A)}(B) \]
    (Proposition \ref{prop:algebras-in-modules}, Theorem \ref{thm:base-change-ordinary-modules}).
    
    \item For an algebra $A$ in a monoidal 2-category $\mathfrak{C}$, the full center $\mathbf{Z}_1(A)$ admits two equivalent descriptions: as $\mathbf{V}^R(A)$, the image of $A$ under the right adjoint of the central action 2-functor 
    \[\mathbf{V} \colon \Z_1(\mathfrak{C}) \to \Bimod_\mathfrak{C}(A), \] 
    and, when the counit component is epic, as the left center $\Zl(\mathbf{U}^R(A))$ of the right adjoint of the forgetful 2-functor $\mathbf{U} \colon \Z_1(\mathfrak{C}) \to \mathfrak{C}$ (Theorem \ref{prop:full-center}).
    
    \item The full center is Morita invariant: Morita equivalent algebras have equivalent full centers as braided algebras in $\Z_1(\mathfrak{C})$ (Corollary \ref{cor:full_center_is_morita_invariant}).
\end{enumerate}

The description of the full center through the right adjoint of the forgetful 2-functor is what makes the one-sided centers computable in practice, and it is the mechanism behind the Morita invariance in (4).

\subsection{Examples}

The final section connects the theory back to familiar examples in the field. When $\mathfrak{B} = \mathbf{2Vect}$, algebras are finite semisimple monoidal categories and the left, right, and full centers all coincide with the Drinfeld center \cite{Drinfeld,Majid,JS91}, recovering the classical Morita invariance of the center of a fusion category \cite{ENO05}. When $\mathfrak{C} = \mathbf{2Vect}_G$ (or its $\pi$-twisted variant), the full center of a $G$-graded fusion category is its $G$-crossed braided Drinfeld center \cite{GNN09,TV13}, and its Morita invariance recovers a known result \cite{ENO10,GJS21}; here the grading separates the one-sided centers from the full center, in contrast to the ungraded case. Finally, for $\mathfrak{C} = \mathbf{Mod}(\mathcal{B})$ over a braided fusion category $\mathcal{B}$, algebras are central $\mathcal{B}$-module monoidal categories in the sense of Davydov--Nikshych \cite{DN13,DN21}, whereas the full center recovers the ordinary Drinfeld center of the underlying monoidal category.

\subsection*{Acknowledgements}

The author is supported by Villum Fonden 00060714 “Global Categorical Symmetries and Phases of Quantum Matter”.
\section{Preliminaries}\label{sec:preliminaries}

\subsection{2-categorical foundations}\label{subsec:foundations}

Throughout this section, we take the \emph{semistrict} model of monoidal 2-categories, equivalently, \emph{Gray monoids}. A monoidal 2-category $\mathfrak{C}$ is called semistrict if all unitors are strict, and the associator is strict on the level of objects. The only coherence data surviving in a semistrict monoidal 2-category are the \emph{interchangers} of 1-morphisms.

We use the following form of the Gordon--Power--Street coherence theorem.

\begin{theorem}[Gordon--Power--Street]\label{thm:coherence_of_tricategories}
Every (weak) 3-category is equivalent to a category enriched in the category $\mathbf{Gray}$ of strict 2-categories, strict 2-functors and Gray tensor product. Consequently, after delooping, every monoidal 2-category may be replaced by a semistrict monoidal 2-category, or Gray monoid.
\end{theorem}

Thus all string diagram arguments below are carried out in the semistrict model. We use the graphical conventions of the author's thesis \cite{Xu-thesis}:
\begin{itemize}[leftmargin=2em]
    \item Regions are labelled by objects, strands by 1-morphisms, nodes by 2-morphisms.
    
    \item Region labels are usually omitted. They are determined by the source and target objects of the neighbouring strands.
    
    \item Strands are drawn horizontally and composed vertically from top to bottom.
    
    \item Horizontal juxtaposition of nodes is from left to right.
    
    \item Identity 1-morphisms are drawn as transparent strands, and identity 2-morphisms as transparent vertices.
\end{itemize}

\begin{notation}\label{not:diagram_calculus_interchanger}
We denote the interchanger as follows. For 1-morphisms $f:x\to y$ and $g:u\to v$, it is an invertible 2-morphism 
$(f\boxt v)\circ(x\boxt g)
    \to
    (y\boxt g)\circ(f\boxt u)$ denoted by the following string diagram:
\[
    \pmb{\phi}_{f,g} \quad := \quad \stringdiagramfigure{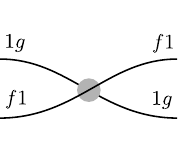} \quad ;
\]
similarly, its inverse is denoted by
\[
    \pmb{\phi}_{f,g}^{-1} \quad := \quad \stringdiagramfigure{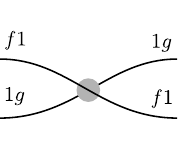} \quad .
\]
Here the label $1$ denotes the identity, indicating the tensor factor that is held fixed. We usually suppress the symbol $\Box$ in the diagram itself. The interchangers record the exchange between left and right tensoring by 1-morphisms. They are natural in $f$ and $g$, preserve identity 1-morphisms, and are compatible with both composition and tensor product of 1-morphisms.
\end{notation}

\begin{notation}\label{not:diagram_calculus_naturality}
We use the following graphical convention for the naturality 2-isomorphism of a 2-natural transformation $\tau \colon F \to G$ between strict 2-categories $\mathfrak{C}$ and $\mathfrak{D}$. For a 1-morphism $f:x\to y$ in $\mathfrak{C}$, the naturality of $\tau$ is witnessed by an invertible 2-morphism $G(f) \circ \tau_x \to \tau_y \circ F(f)$, denoted by
\[\tau_f \quad := \quad \stringdiagramfigure{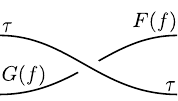} \quad ;\]
its inverse is denoted by
\[\tau_f^{-1} \quad := \quad \stringdiagramfigure{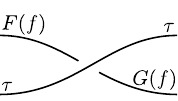} \quad .\]
\end{notation}

We also need to work with braided monoidal 2-categories in this article. We use the following form of coherence theorem from Gurski \cite{Gur2}.

\begin{theorem}[Gurski]\label{thm:gurski_braided_coherence}
Every braided monoidal 2-category is equivalent to a semistrict braided monoidal 2-category. In particular, after passing to a semistrict model, the underlying monoidal 2-category is a Gray monoid and the braiding is encoded by a 2-natural isomorphism $b$ together with two modifications $R$ and $S$, subject to the coherence conditions below.
\end{theorem}

\begin{notation}\label{not:braiding_calculus}
Let $\mathfrak{B}$ be a semistrict braided monoidal 2-category. For objects $x,y,z$ of $\mathfrak{B}$:
\begin{itemize}[leftmargin=2em]
    \item The braiding is a 2-natural isomorphism
    \[
        b_{x,y}:x\boxt y\to y\boxt x.
    \]

    \item The modification $R_{x,y,z}$ compares the braiding with the monoidal product in the second variable:
    \[
        R_{x,y,z}:(b_{x,y}\boxt z)\circ(y\boxt b_{x,z})
        \to b_{x,y\boxt z}.
    \]

    \item The modification $S_{x,y,z}$ compares the braiding with the monoidal product in the first variable:
    \[
        S_{x,y,z}:(x\boxt b_{y,z})\circ(b_{x,z}\boxt y)
        \to b_{x\boxt y,z}.
    \]
\end{itemize}
The modifications $R$ and $S$ satisfy the following coherence equations.
\begin{enumerate}[label=(\alph*),leftmargin=2em]
    \item For all objects $x,y,z,w$ in $\mathfrak{B}$, the equation
    \[
    \begin{array}{@{}c@{\quad}c@{\quad}c@{}}
        \stringdiagramfigure{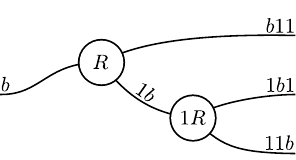}
        &
        =
        &
        \stringdiagramfigure{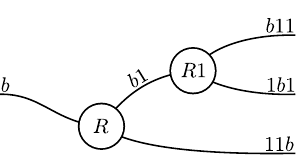}
    \end{array}
    \]
    holds in $\mathbf{Hom}_{\mathfrak{B}}(x\boxt y\boxt z\boxt w,\, y\boxt z\boxt w\boxt x)$;

    \item For all objects $x,y,z,w$ in $\mathfrak{B}$, the equation
    \[
    \begin{array}{@{}c@{\quad}c@{\quad}c@{}}
        \stringdiagramfigure{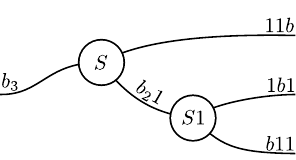}
        &
        =
        &
        \stringdiagramfigure{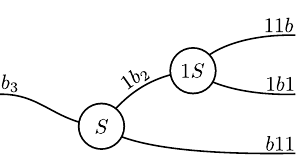}
    \end{array}
    \]
    holds in $\mathbf{Hom}_{\mathfrak{B}}(x\boxt y\boxt z\boxt w,\, w\boxt x\boxt y\boxt z)$;

    \item For all objects $x,y,z,w$ in $\mathfrak{B}$, the equation
    \[
    \begin{array}{@{}c@{\quad}c@{\quad}c@{}}
        \stringdiagramfigure{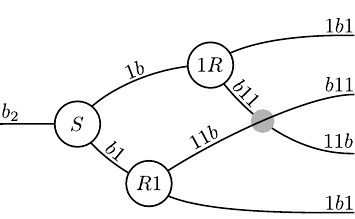}
        &
        =
        &
        \stringdiagramfigure{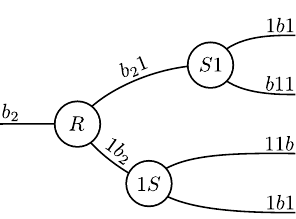}
    \end{array}
    \] 
    holds in $\mathbf{Hom}_{\mathfrak{B}}(x\boxt y\boxt z\boxt w,\, z\boxt w\boxt x\boxt y)$;

    \item For all objects $x,y,z$ in $\mathfrak{B}$, the equation
    \[
    \begin{array}{@{}c@{\quad}c@{\quad}c@{}}
        \stringdiagramfigure{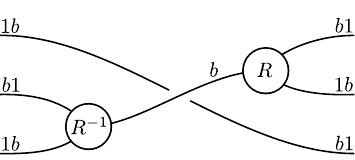}
        &
        =
        &
        \stringdiagramfigure{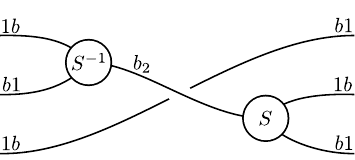}
    \end{array}
    \] 
    holds in $\mathbf{Hom}_{\mathfrak{B}}(x\boxt y\boxt z,\, z\boxt y\boxt x)$;

    \item The following 2-natural isomorphisms are identities: \[b_{-,\I} : \quad x \mapsto b_{x,\I}:x \equiv x \, \Box \,  \I\rightarrow \I \, \Box \, x \equiv x, \] \[ b_{\I,-} : \quad x \mapsto b_{\I,x}: x \equiv \I \, \Box \, x \rightarrow x \, \Box \, \I \equiv x; \]
        
    \item For objects $x,y,z$ in $\mathfrak{B}$, the 2-isomorphisms $R_{x,y,z}$ and $S_{x,y,z}$ are identities whenever one of $x$, $y$, and $z$ is equal to $\I$.
        
\end{enumerate}

More generally in the diagram calculus of braided monoidal 2-categories, for the braiding of $(n+m)$ objects, the notation $b_n$ indicates the braiding of the first $n$ objects over the subsequent $m$ objects. To prevent ambiguity with numerical indices, we write $b$ instead of $b_1$ when $n=1$.
\end{notation}

\begin{notation}\label{not:monoidal_opposite}
    For a monoidal 2-category $\mathfrak{C}$, we denote its \textit{monoidal reversal} by $\mathfrak{C}^{1mp}$, which has the same underlying 2-category but the direction of monoidal product is reversed (along with other coherence data). One can also think of monoidal reversal being the delooping 3-category with the composition of 1-morphisms in the opposite direction, i.e. $\mathrm{B} \mathfrak{C}^{1mp} = (\mathrm{B} \mathfrak{C})^{1op}$. 

    Taking monoidal reversal is functorial, i.e. for monoidal 2-functor $F:\mathfrak{C} \to \mathfrak{D}$, there is a monoidal 2-functor between the monoidal reversals $F^{1mp}:\mathfrak{C}^{1mp} \to \mathfrak{D}^{1mp}$, with the same underlying 2-functor and a canonically induced monoidal 2-functor structure.
\end{notation}

\begin{notation}\label{not:braided_opposite}
    For a braided monoidal 2-category $\mathfrak{B}$, we denote its \textit{braided reversal} by $\mathfrak{B}^{2mp}$, which has the same underlying monoidal 2-category but the direction of the braiding is reversed (along with other coherence data). One can also think of braided reversal being the double delooping 4-category with the composition of 2-morphisms in the opposite direction, i.e. $\mathrm{B}^2 \mathfrak{B}^{2mp} = (\mathrm{B}^2 \mathfrak{B})^{2op}$.

    Taking braided reversal is functorial, i.e. for braided monoidal 2-functor $F:\mathfrak{A} \to \mathfrak{B}$, there is a braided monoidal 2-functor between the braided reversals $F^{2mp}:\mathfrak{A}^{2mp} \to \mathfrak{B}^{2mp}$, with the same underlying monoidal 2-functor and a canonically induced braided monoidal 2-functor structure.
\end{notation}

\subsection{Algebras}\label{subsec:algebras}
Let $(\mathfrak{C},\Box,\I,\pmb{\phi})$ be a semistrict monoidal 2-category.

\begin{definition}
An \emph{algebra} in $\mathfrak{C}$ consists of an object $A$, multiplication and unit 1-morphisms
\begin{equation}
    m \colon A\boxt A\to A,
    \qquad
    i \colon \I\to A,
\end{equation}
associator and unitor 2-isomorphisms
\begin{equation}
    \alpha \colon m \circ (m\boxt 1_A)\to m \circ (1_A\boxt m),
    \qquad
    \lambda \colon m \circ (i\boxt 1_A)\to 1_A,
    \qquad
    \rho \colon m \circ (1_A\boxt i)\to 1_A,
\end{equation}
subject to the pentagon and triangle coherence equations. The pentagon is the following equality in
$\mathbf{Hom}_{\mathfrak{C}}(A\boxt A\boxt A\boxt A,A)$:
\begin{equation}\label{eqn:AlgebraAssociativity}
    \begin{array}{@{}c@{\quad}c@{\quad}c@{}}
        \stringdiagramfigure{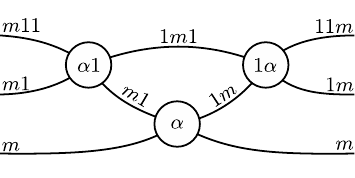}
        &
        =
        &
        \stringdiagramfigure{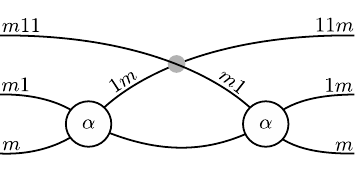} \quad .
    \end{array}
\end{equation}
The triangle is the following equality in
$\mathbf{Hom}_{\mathfrak{C}}(A\boxt A,A)$:
\begin{equation}\label{eqn:AlgebraUnitality}
    \begin{array}{@{}c@{\quad}c@{\quad}c@{}}
        \stringdiagramfigure{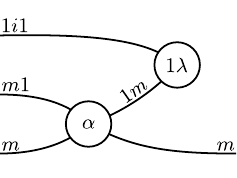}
        &
        =
        &
        \stringdiagramfigure{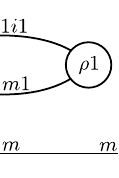} \quad .
    \end{array}
\end{equation}

\end{definition}

\begin{definition}\label{def:algebra_1morphism}
Let $A$ and $B$ be algebras in $\mathfrak{C}$. An \emph{algebra 1-morphism} $f:A\to B$ is a 1-morphism in $\mathfrak{C}$ equipped with invertible 2-morphisms
\begin{equation}
    \psi^f:m^B\circ(f\boxt f)\to f\circ m^A,
    \qquad
    \eta^f:i^B\to f\circ i^A,
\end{equation}
compatible with the associators and unitors of $A$ and $B$. The associativity condition is the following equality in
$\mathbf{Hom}_{\mathfrak{C}}(A\boxt A\boxt A,B)$:
\begin{equation}\label{eqn:Algebra1MorphismCoh1}
    \begin{array}{@{}c@{\quad}c@{\quad}c@{}}
        \stringdiagramfigure{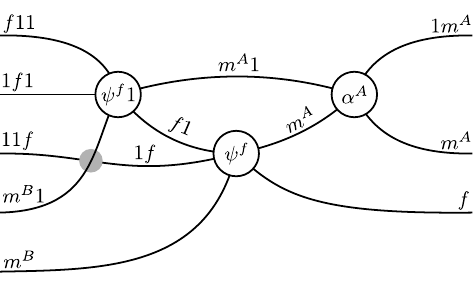}
        &
        =
        &
        \stringdiagramfigure{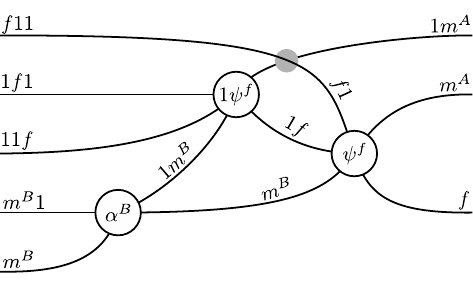} \quad .
    \end{array}
\end{equation}
The left unitality condition is the following equality in
$\mathbf{Hom}_{\mathfrak{C}}(A,B)$:
\begin{equation}\label{eqn:Algebra1MorphismCoh2}
    \begin{array}{@{}c@{\quad}c@{\quad}c@{}}
        \stringdiagramfigure{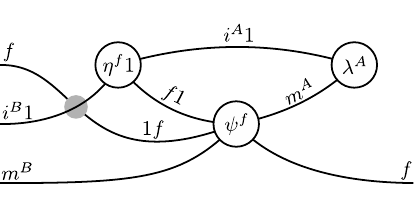}
        &
        =
        &
        \stringdiagramfigure{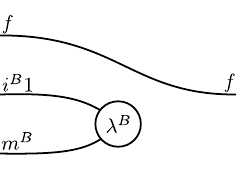} \quad .
    \end{array}
\end{equation}
The right unitality condition is the following equality in
$\mathbf{Hom}_{\mathfrak{C}}(A,B)$:
\begin{equation}\label{eqn:Algebra1MorphismCoh3}
    \begin{array}{@{}c@{\quad}c@{\quad}c@{}}
        \stringdiagramfigure{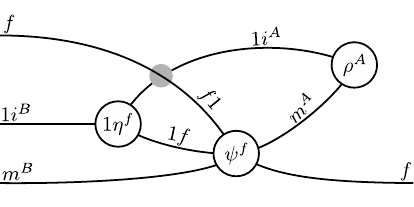}
        &
        =
        &
        \stringdiagramfigure{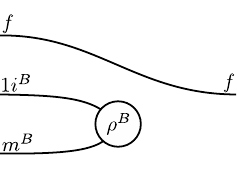} \quad .
    \end{array}
\end{equation}
\end{definition}

\begin{definition}\label{def:algebra_2morphism}
Let $A$ and $B$ be algebras in $\mathfrak{C}$, and let $f,g:A\to B$ be algebra 1-morphisms. An \emph{algebra 2-morphism} $\gamma:f\to g$ is a 2-morphism in $\mathfrak{C}$ compatible with the structure 2-morphisms of $f$ and $g$. The multiplication compatibility is the following equality in
$\mathbf{Hom}_{\mathfrak{C}}(A\boxt A,B)$:
\begin{equation}\label{eqn:Algebra2MorphismCoh1}
    \begin{array}{@{}c@{\quad}c@{\quad}c@{}}
        \stringdiagramfigure{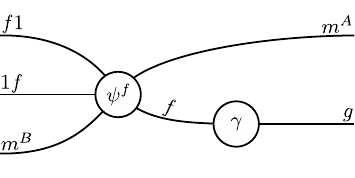}
        &
        =
        &
        \stringdiagramfigure{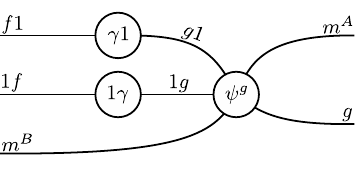} \quad .
    \end{array}
\end{equation}
The unit compatibility is the following equality in
$\mathbf{Hom}_{\mathfrak{C}}(\I,B)$:
\begin{equation}\label{eqn:Algebra2MorphismCoh2}
    \begin{array}{@{}c@{\quad}c@{\quad}c@{}}
        \stringdiagramfigure{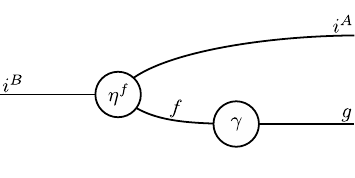}
        &
        =
        &
        \stringdiagramfigure{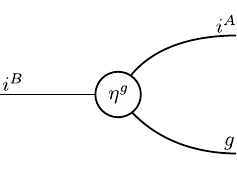} \quad .
    \end{array}
\end{equation}
\end{definition}

\begin{notation}
We write $\mathbf{Alg}(\mathfrak{C})$ for the resulting 2-category of algebras, algebra 1-morphisms, and algebra 2-morphisms.
\end{notation}

\subsection{Braided algebras}\label{subsec:braided_algebras}

Let $(\mathfrak{B},\Box,\I,\pmb{\phi},b,R,S)$ be a braided semistrict monoidal 2-category.

\begin{definition}
A \emph{braided algebra} in $\mathfrak{B}$ is an algebra $B$ together with a 2-isomorphism
\begin{equation}
    \beta:m\circ b_{B,B}\to m
\end{equation}
satisfying the two hexagon-type compatibility conditions with multiplication. Explicitly, the first hexagon is the following equality in
$\mathbf{Hom}_{\mathfrak{B}}(B\boxt B\boxt B,B)$:
\begin{equation}\label{eqn:BraidedAlgebra1}
    \begin{array}{@{}c@{\quad}c@{\quad}c@{}}
        \stringdiagramfigure{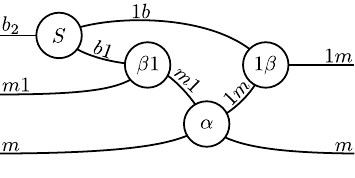}
        &
        =
        &
        \stringdiagramfigure{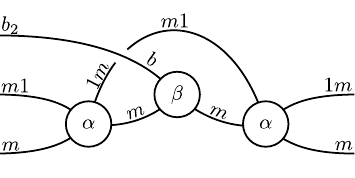} \quad .
    \end{array}
\end{equation}
The second hexagon is the following equality in
$\mathbf{Hom}_{\mathfrak{B}}(B\boxt B\boxt B,B)$:
\begin{equation}\label{eqn:BraidedAlgebra2}
    \begin{array}{@{}c@{\quad}c@{\quad}c@{}}
        \stringdiagramfigure{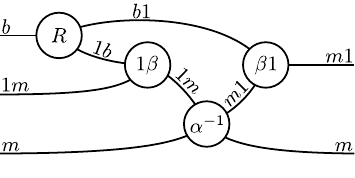}
        &
        =
        &
        \stringdiagramfigure{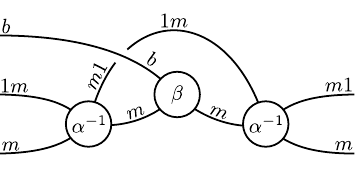} \quad .
    \end{array}
\end{equation}
Let $A$ and $B$ be braided algebras in $\mathfrak{B}$. A \emph{braided algebra 1-morphism} $f:A\to B$ is an algebra 1-morphism compatible with the braiding data. Explicitly, it satisfies the following equality in
$\mathbf{Hom}_{\mathfrak{B}}(A\boxt A,B)$:
\begin{equation}\label{eqn:BraidedAlgebra1Morphism}
    \begin{array}{@{}c@{\quad}c@{\quad}c@{}}
        \stringdiagramfigure{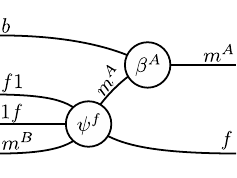}
        &
        =
        &
        \stringdiagramfigure{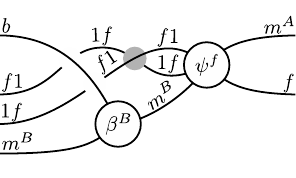} \quad .
    \end{array}
\end{equation}
\end{definition}

\subsection{Modules and bimodules}\label{subsec:modules}
Let $(\mathfrak{C},\Box,\I,\pmb{\phi})$ be a semistrict monoidal 2-category.

\begin{definition}\label{def:left_module}
Let $(A,m^A,i^A,\alpha^A,\lambda^A,\rho^A)$ be an algebra in $\mathfrak{C}$. A \emph{left $A$-module} is an object $M$ together with an action 1-morphism
\begin{equation}
    l^M:A\boxt M\to M,
\end{equation}
associator and unitor 2-isomorphisms
\begin{equation}
    \mu^M \colon l^M \circ (m^A \boxt 1_M)\to l^M \circ (1_A \boxt l^M),
    \qquad
    \lambda^M \colon l^M \circ (i^A \boxt 1_M)\to 1_M,
\end{equation}
satisfying the module pentagon and triangle equations. The pentagon is the following equality in
$\mathbf{Hom}_{\mathfrak{C}}(A\boxt A\boxt A\boxt M,M)$:
\begin{equation}\label{eqn:LeftModAssociativity}
    \begin{array}{@{}c@{\quad}c@{\quad}c@{}}
        \stringdiagramfigure{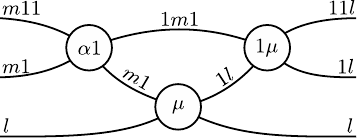}
        &
        =
        &
        \stringdiagramfigure{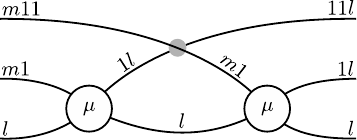} \quad .
    \end{array}
\end{equation}
The triangle is the following equality in
$\mathbf{Hom}_{\mathfrak{C}}(A\boxt M,M)$:
\begin{equation}\label{eqn:LeftModUnitality}
    \begin{array}{@{}c@{\quad}c@{\quad}c@{}}
        \stringdiagramfigure{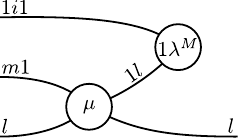}
        &
        =
        &
        \stringdiagramfigure{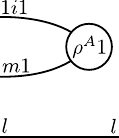} \quad .
    \end{array}
\end{equation}
\end{definition}

\begin{definition}\label{def:left_module_1morphism}
Let $A$ be an algebra in $\mathfrak{C}$, and let $M$ and $N$ be left $A$-modules. A \emph{left $A$-module 1-morphism} $f:M\to N$ is a 1-morphism in $\mathfrak{C}$ equipped with a 2-isomorphism
\begin{equation}
    \psi^f:l^N\circ(1_A\boxt f)\to f\circ l^M,
\end{equation}
compatible with the associators and unitors of $M$ and $N$. The associativity compatibility is the following equality in
$\mathbf{Hom}_{\mathfrak{C}}(A\boxt A\boxt M,N)$:
\begin{equation}\label{eqn:LeftMod1MorphismAssociativity}
    \begin{array}{@{}c@{\quad}c@{\quad}c@{}}
        \stringdiagramfigure{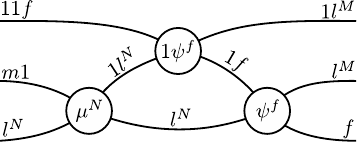}
        &
        =
        &
        \stringdiagramfigure{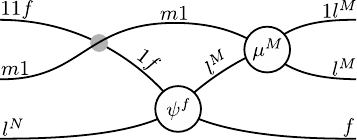} \quad .
    \end{array}
\end{equation}
The unitality compatibility is the following equality in
$\mathbf{Hom}_{\mathfrak{C}}(M,N)$:
\begin{equation}\label{eqn:LeftMod1MorphismUnitality}
    \begin{array}{@{}c@{\quad}c@{\quad}c@{}}
        \stringdiagramfigure{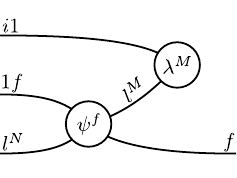}
        &
        =
        &
        \stringdiagramfigure{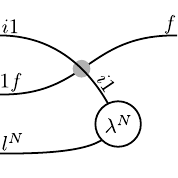} \quad .
    \end{array}
\end{equation}
\end{definition}

\begin{definition}\label{def:left_module_2morphism}
Let $A$ be an algebra in $\mathfrak{C}$, and let $M$ and $N$ be left $A$-modules. Let $f,g:M\to N$ be left $A$-module 1-morphisms. A \emph{left $A$-module 2-morphism} $\xi \colon f\to g$ is a 2-morphism in $\mathfrak{C}$ compatible with the structure 2-morphisms of $f$ and $g$. The action compatibility is the following equality in
$\mathbf{Hom}_{\mathfrak{C}}(A\boxt M,N)$:
\begin{equation}\label{eqn:LeftMod2MorphismCoh}
    \begin{array}{@{}c@{\quad}c@{\quad}c@{}}
        \stringdiagramfigure{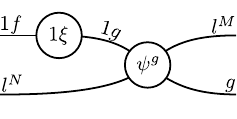}
        &
        =
        &
        \stringdiagramfigure{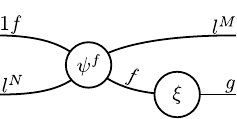} \quad .
    \end{array}
\end{equation}
\end{definition}

\begin{definition}\label{def:right_module}
Similarly, a \emph{right $A$-module} consists of an object $M$ in $\mathfrak{C}$, an action 1-morphism
\begin{equation}
    r^M:M\boxt A\to M,
\end{equation}
associator and unitor 2-isomorphisms
\begin{equation}
    \nu^M \colon r^M \circ (r^M \boxt 1_A)\to r^M \circ (1_M \boxt m^A),
    \qquad
    \rho^M \colon r^M \circ (1_M \boxt i^A)\to 1_M,
\end{equation}
satisfying the module pentagon and triangle equations. The pentagon is the following equality in
$\mathbf{Hom}_{\mathfrak{C}}(M\boxt A\boxt A\boxt A,M)$:
\begin{equation}\label{eqn:RightModAssociativity}
    \begin{array}{@{}c@{\quad}c@{\quad}c@{}}
        \stringdiagramfigure{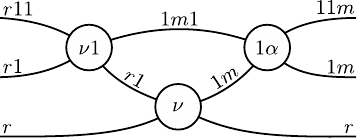}
        &
        =
        &
        \stringdiagramfigure{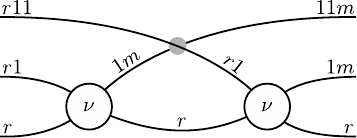} \quad .
    \end{array}
\end{equation}
Its triangle is the following equality in
$\mathbf{Hom}_{\mathfrak{C}}(M\boxt A,M)$:
\begin{equation}\label{eqn:RightModUnitality}
    \begin{array}{@{}c@{\quad}c@{\quad}c@{}}
        \stringdiagramfigure{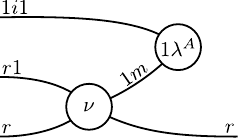}
        &
        =
        &
        \stringdiagramfigure{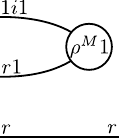} \quad .
    \end{array}
\end{equation}
\end{definition}

\begin{definition}\label{def:right_module_1morphism}
Let $A$ be an algebra in $\mathfrak{C}$, and let $M$ and $N$ be right $A$-modules. A \emph{right $A$-module 1-morphism} $f:M\to N$ is a 1-morphism in $\mathfrak{C}$ equipped with a 2-isomorphism
\begin{equation}
    \psi^f:r^N\circ(f\boxt 1_A)\to f\circ r^M,
\end{equation}
compatible with the associators and unitors of $M$ and $N$. The associativity compatibility is the following equality in
$\mathbf{Hom}_{\mathfrak{C}}(M\boxt A\boxt A,N)$:
\begin{equation}\label{eqn:RightMod1MorphismAssociativity}
    \begin{array}{@{}c@{\quad}c@{\quad}c@{}}
        \stringdiagramfigure{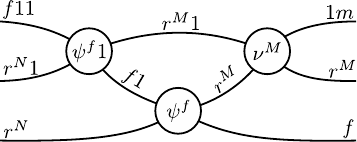}
        &
        =
        &
        \stringdiagramfigure{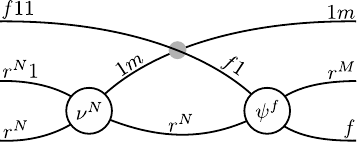} \quad .
    \end{array}
\end{equation}
The unitality compatibility is the following equality in
$\mathbf{Hom}_{\mathfrak{C}}(M,N)$:
\begin{equation}\label{eqn:RightMod1MorphismUnitality}
    \begin{array}{@{}c@{\quad}c@{\quad}c@{}}
        \stringdiagramfigure{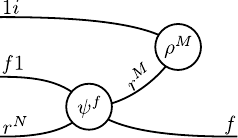}
        &
        =
        &
        \stringdiagramfigure{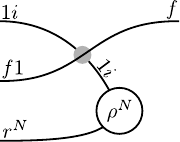} \quad .
    \end{array}
\end{equation}
\end{definition}

\begin{definition}\label{def:right_module_2morphism}
Let $A$ be an algebra in $\mathfrak{C}$, and let $M$ and $N$ be right $A$-modules. Let $f,g:M\to N$ be right $A$-module 1-morphisms. A \emph{right $A$-module 2-morphism} $\xi \colon f\to g$ is a 2-morphism in $\mathfrak{C}$ compatible with the structure 2-morphisms of $f$ and $g$. The action compatibility is the following equality in
$\mathbf{Hom}_{\mathfrak{C}}(M\boxt A,N)$:
\begin{equation}\label{eqn:RightMod2MorphismCoh}
    \begin{array}{@{}c@{\quad}c@{\quad}c@{}}
        \stringdiagramfigure{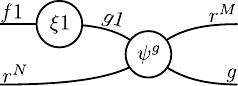}
        &
        =
        &
        \stringdiagramfigure{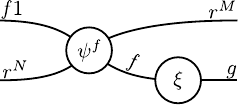} \quad .
    \end{array}
\end{equation}
\end{definition}

\begin{definition}
Let $A$ and $B$ be algebras. An \emph{$(A,B)$-bimodule} is an object $M$ with a left $A$-module structure $(l^M,\mu^M,\lambda^M)$, a right $B$-module structure $(r^M,\nu^M,\rho^M)$, and an extra 2-isomorphism
\begin{equation}
    \alpha^M \colon r^M \circ (l^M \boxt 1_B) \to l^M \circ (1_A \boxt r^M),
\end{equation}
satisfying the two extra pentagon coherence equations.

The left-action compatibility pentagon is the following equality in
$\mathbf{Hom}_{\mathfrak{C}}(A\boxt A\boxt M\boxt B,M)$:
\begin{equation}\label{eqn:BiModAssociativity1}
    \begin{array}{@{}c@{\quad}c@{\quad}c@{}}
        \stringdiagramfigure{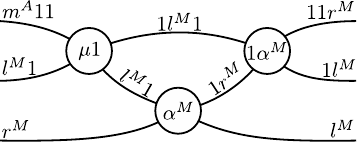}
        &
        =
        &
        \stringdiagramfigure{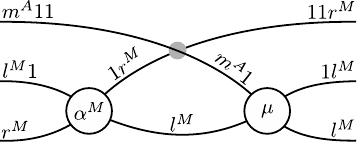} \quad .
    \end{array}
\end{equation}
The right-action compatibility pentagon is the following equality in
$\mathbf{Hom}_{\mathfrak{C}}(A\boxt M\boxt B\boxt B,M)$:
\begin{equation}\label{eqn:BiModAssociativity2}
    \begin{array}{@{}c@{\quad}c@{\quad}c@{}}
        \stringdiagramfigure{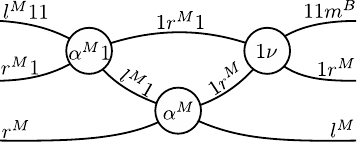}
        &
        =
        &
        \stringdiagramfigure{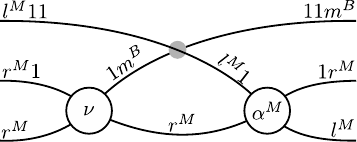} \quad .
    \end{array}
\end{equation}
\end{definition}

\begin{definition}\label{def:bimodule_1morphism}
Let $A$ and $B$ be algebras in $\mathfrak{C}$, and let $M$ and $N$ be $(A,B)$-bimodules. An \emph{$(A,B)$-bimodule 1-morphism} $f:M\to N$ is a 1-morphism in $\mathfrak{C}$ that is simultaneously equipped with a left $A$-module 1-morphism structure $\varphi^f$ and a right $B$-module 1-morphism structure $\psi^f$, satisfying the extra bimodule compatibility condition in
$\mathbf{Hom}_{\mathfrak{C}}(A\boxt M\boxt B,N)$:
\begin{equation}\label{eqn:Bimod1MorphismCoh}
    \begin{array}{@{}c@{\quad}c@{\quad}c@{}}
        \stringdiagramfigure{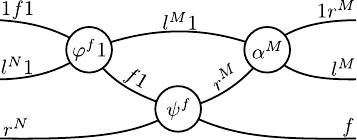}
        &
        =
        &
        \stringdiagramfigure{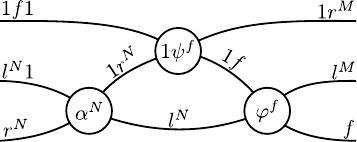} \quad .
    \end{array}
\end{equation}
\end{definition}

\begin{remark}
    2-morphisms between bimodule 1-morphisms are just 2-morphisms in $\mathfrak{C}$ compatible with the left and right module 1-morphism structures.
\end{remark}

\begin{notation}
We denote by $\LMod_{\mathfrak{C}}(A)$, $\Mod_{\mathfrak{C}}(A)$, and $\Bimod_{\mathfrak{C}}(A,B)$ the 2-categories of left $A$-modules, right $A$-modules, and $(A,B)$-bimodules in $\mathfrak{C}$, respectively.
\end{notation}

\subsection{Relative tensor products}
Let $(\mathfrak{C},\Box,\I,\pmb{\phi})$ be a semistrict monoidal 2-category.

\begin{definition}\label{def:balanced_1morphism}
Let $A$ be an algebra in $\mathfrak{C}$, $M$ be a right $A$-module and $N$ a left $A$-module. A \emph{$A$-balanced 1-morphism} from $M\boxt N$ to an object $L$ consists of a 1-morphism $f \colon M\boxt N\to L$ together with a 2-isomorphism
\begin{equation}
    \tau^f \colon f\circ(r^M \boxt 1_N) \to f \circ (1_M \boxt l^N),
\end{equation}
satisfying associativity and unitality coherence conditions.
The associativity condition is the pentagon equality in
$\mathbf{Hom}_{\mathfrak{C}}(M\boxt A\boxt A\boxt N,L)$:
\begin{equation}\label{eqn:Balanced1MorAssociativity}
    \begin{array}{@{}c@{\quad}c@{\quad}c@{}}
        \stringdiagramfigure{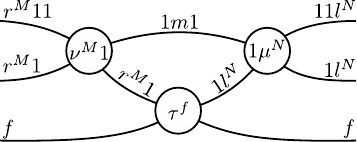}
        &
        =
        &
        \stringdiagramfigure{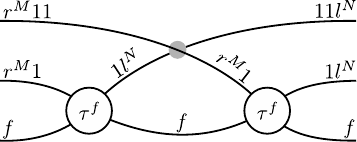} \quad .
    \end{array}
\end{equation}
The unitality condition is the triangle equality in
$\mathbf{Hom}_{\mathfrak{C}}(M\boxt N,L)$:
\begin{equation}\label{eqn:Balanced1MorUnitality}
    \begin{array}{@{}c@{\quad}c@{\quad}c@{}}
        \stringdiagramfigure{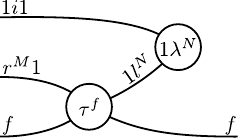}
        &
        =
        &
        \stringdiagramfigure{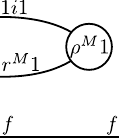} \quad .
    \end{array}
\end{equation}
\end{definition}

\begin{definition}\label{def:balanced_2morphism}
Let $f,g \colon M \boxt N \to L$ be two $A$-balanced 1-morphisms. A \emph{$A$-balanced 2-morphism} from $f$ to $g$ consists of a 2-morphism $\xi \colon f \to g$ satisfying the following equality in
$\mathbf{Hom}_{\mathfrak{C}}(M\boxt A\boxt N,L)$:
\begin{equation}\label{eqn:Balanced2Mor}
    \begin{array}{@{}c@{\quad}c@{\quad}c@{}}
        \stringdiagramfigure{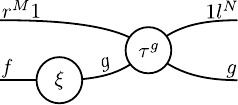}
        &
        =
        &
        \stringdiagramfigure{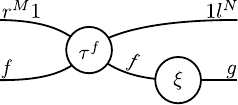} \quad .
    \end{array}
\end{equation}
\end{definition}

\begin{definition}\label{def:relative_tensor_product}
The \emph{relative tensor product} $M\boxtsub{A}N$, when it exists, is an object equipped with a universal $A$-balanced 1-morphism
\begin{equation}\label{eqn:relative_tensor_product_balanced_1morphism}
    t_{M,N}\colon M\boxt N\to M\boxtsub{A}N.
\end{equation}
Equivalently, it co-represents the 2-functor $\mathfrak{C} \to \mathbf{2Vect}$ which sends any object $L$ to the category of $A$-balanced morphisms $M\boxt N \to L$.
\end{definition}

\begin{proposition}\label{prop:associativity_and_unitality_of_relative_tensor_products}
Assume the relevant relative tensor products exist. For algebras $A,B,C,D$, an $(A,B)$-bimodule $M$, a $(B,C)$-bimodule $N$, and a $(C,D)$-bimodule $P$, there is a canonical $(A,D)$-bimodule 1-isomorphism
\begin{equation}\label{eqn:relative_tensor_product_associator}
    (M\boxtsub{B}N)\boxtsub{C}P
    \simeq
    M\boxtsub{B}(N\boxtsub{C}P),
\end{equation}
and $(A,B)$-bimodules 1-isomorphisms 
\begin{equation}\label{eqn:relative_tensor_product_unitor}
    A\boxtsub{A}M\simeq M \quad \text{and} \quad M\boxtsub{B}B\simeq M.
\end{equation}
\end{proposition}

\begin{definition}\label{def:Morita_equivalence}
Two algebras $A$ and $B$ in $\mathfrak{C}$ are \emph{Morita equivalent} if there are bimodules ${}_A M_B$ and ${}_B N_A$ together with bimodule 1-isomorphisms $M\boxtsub{B}N\simeq A$ and $N\boxtsub{A}M\simeq B$.
\end{definition}

\section{Monoidal 2-Adjunctions}\label{sec:monoidal_2-adjunction}
In this section, we establish the categorification of the well-known 1-categorical result that the left adjoint to a lax monoidal functor is oplax monoidal, and the right adjoint to an oplax monoidal functor is lax monoidal. This general framework will be crucial for our later analysis of base change for module 2-categories and the construction of full centers.

\subsection{Monoidal 2-functors}\label{subsec:monoidal_2-functor}
Suppose $\mathfrak{C}$ and $\mathfrak{D}$ are two semistrict monoidal 2-categories.

\begin{definition} \label{def:monoidal_2-functor}
    A \textit{monoidal 2-functor} from $\mathfrak{C}$ to $\mathfrak{D}$ consists of:
    \begin{enumerate}
        \item[1.] An underlying 2-functor\footnote{We omit the unital constraints of the underlying 2-functor $F$ in the following diagrams. One can also assume that we start with those \textit{strictly unital} 2-functors.} $F:\mathfrak{C} \to \mathfrak{D}$;

        \item[2.] A 2-natural isomorphism $\chi_{x,y}:F(x) \, \Box^\mathfrak{D} \, F(y) \to F(x \, \Box^\mathfrak{C} \, y)$ given on objects $x,y$ in $\mathfrak{C}$;

        \item[3.] A 1-isomorphism $\iota:\mathbf{I}^\mathfrak{D} \to F(\mathbf{I}^\mathfrak{C})$;

        \item[4.] Invertible modifications
    \end{enumerate}
    \[\begin{tikzcd}[column sep=45pt,row sep=25pt]
            {F(x) \, \Box^\mathfrak{D} \, F(y) \, \Box^\mathfrak{D} \, F(z)}
                \arrow[r,"\chi_{x,y} \, \Box^\mathfrak{D} \, F(z)"]
                \arrow[d,"F(x) \, \Box^\mathfrak{D} \, \chi_{y,z}"']
            & {F(x \, \Box^\mathfrak{C} \, y) \, \Box^\mathfrak{D} \, F(z)}
                \arrow[d,"\chi_{xy,z}"]
                \arrow[dl,Rightarrow,shorten <=25pt, shorten >=25pt,"\omega_{x,y,z}"']
            \\ {F(x) \, \Box^\mathfrak{D} \, F(y \, \Box^\mathfrak{C} \, z)}
                \arrow[r,"\chi_{x,yz}"']
            & {F(x \, \Box^\mathfrak{C} \, y \, \Box^\mathfrak{C} \, z)}
        \end{tikzcd}, \]
        \[\begin{tikzcd}[column sep=40pt]
            {F(\mathbf{I}^\mathfrak{C}) \, \Box^\mathfrak{D} \, F(x)}
                \arrow[dd,"\chi_{\mathbf{I},x}"']
            & {}
            & {F(x) \, \Box^\mathfrak{D} \, F(\mathbf{I}^\mathfrak{C})}
                \arrow[dd,"\chi_{x,\mathbf{I}}"]
            \\ {}
            & {F(x)}
                \arrow[dl,equal]
                \arrow[dr,equal]
                \arrow[ul,"\iota \, \Box^\mathfrak{D} \, F(x)"']
                \arrow[ur,"F(x) \, \Box^\mathfrak{D} \, \iota"]
                \arrow[l,Leftarrow,"\gamma_x"',shorten <= 30pt, shorten >= 30pt]
                \arrow[r,Leftarrow,"\delta_x",shorten <= 30pt, shorten >= 30pt]
            & {}                
            \\ {F(x)}
            & {}
            & {F(x)}
        \end{tikzcd},\]
        \begin{enumerate}
            \item[] given on objects $x,y,z$ in $\mathfrak{C}$;
        \end{enumerate}
        subject to the following conditions:
        \begin{enumerate}
            \item[a.] For any objects $x,y,z,w$ in $\mathfrak{C}$, the equation holds
        \end{enumerate}

        \begin{equation}\label{eqn:Monoidal2FunctorAssociativity}
        \begin{tabular}{@{}ccc@{}}
    
        \stringdiagramfigure{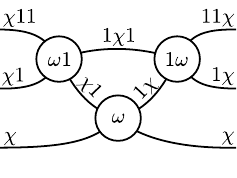} & $=$ &
        \stringdiagramfigure{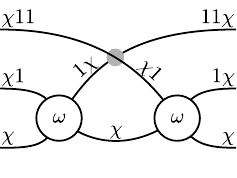} 
        
        \end{tabular}
        \end{equation}

        \begin{enumerate}
            \item[] in $\Hom_\mathfrak{D}(F(x) \, \Box^\mathfrak{D} \, F(y) \, \Box^\mathfrak{D} \, F(z) \, \Box^\mathfrak{D} \, F(w),F(x \, \Box^\mathfrak{C} \, y \, \Box^\mathfrak{C} \, z \, \Box^\mathfrak{C} \, w ))$;
        \end{enumerate}

        \begin{enumerate}
            \item[b.] For any objects $x,y$ in $\mathfrak{C}$, the equation holds
        \end{enumerate}

        \begin{equation}\label{eqn:Monoidal2FunctorUnitality}
        \begin{tabular}{@{}ccc@{}}
    
        \stringdiagramfigure{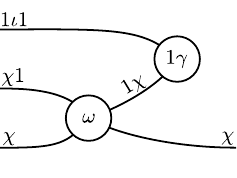} & $=$ &
        \stringdiagramfigure{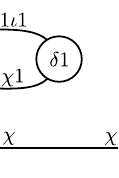} 
        
        \end{tabular}
        \end{equation}

        \begin{enumerate}
            \item[] in $\Hom_\mathfrak{D}(F(x) \, \Box^\mathfrak{D} \, F(y),F(x \, \Box^\mathfrak{C} \, y))$.
        \end{enumerate}
\end{definition}

\begin{notation} \label{not:lax_monoidal_2-functor} 
    In the above definition, if we do not require the 2-natural transformation $\chi$ and 1-morphism $\iota$ to be invertible, then it is called a \textit{lax} monoidal 2-functor, if the directions of $\chi$ and $\iota$ are the same as given, or an \textit{oplax} monoidal 2-functor, if the directions of $\chi$ and $\iota$ are reversed. The monoidal 2-functors we defined above are called \textit{strong} monoidal 2-functors if we need to emphasize that $\chi$ and $\iota$ are invertible.
\end{notation}

\subsection{Monoidal 2-natural Transformations} \label{subsec:monoidal_2NatTrans}

\begin{definition} \label{def:monoidal_2NatTrans}
    Let $F$ and $G$ be two monoidal 2-functors given from $\mathfrak{C}$ to $\mathfrak{D}$. A \textit{monoidal 2-natural transformation} from $F$ to $G$ consists of:
    \begin{enumerate}
        \item[1.] An underlying 2-natural transformation $\eta:F \to G$;

        \item[2.] Invertible modification
        \[\begin{tikzcd}[column sep=40pt,row sep=30pt]
            {F(x) \, \Box^\mathfrak{D} \, F(y)}
                \arrow[r,"\eta_x \, \Box^\mathfrak{D} \, F(y)"]
                \arrow[d,"\chi^F_{x,y}"']
            & {G(x) \, \Box^\mathfrak{D} \, F(y)}
                \arrow[r,"G(x) \, \Box^\mathfrak{D} \eta_y"]
            & {G(x) \, \Box^\mathfrak{D} \, G(y)}
                \arrow[d,"\chi^G_{x,y}"]
                \arrow[lld,Rightarrow,"\Pi_{x,y}"',shorten <=70pt, shorten >=50pt]
            \\ {F(x \, \Box^\mathfrak{C} \, y)}
                \arrow[rr,"\eta_{xy}"']
            & {}
            & {G(x \, \Box^\mathfrak{C} \, y)}
        \end{tikzcd},\] and a 2-isomorphism
        \[\begin{tikzcd}[row sep=30pt]
            {\mathbf{I}^\mathfrak{D}}
                \arrow[rr,"\iota^G"]
                \arrow[rd,"\iota^F"']
            & {}
            & {G(\mathbf{I}^\mathfrak{C})}
            \\ {}
            & {F(\mathbf{I}^\mathfrak{C})}
                \arrow[ur,"\eta_\mathbf{I}"']
                \arrow[u,Rightarrow,shorten <=10pt, shorten >=5pt,"M"]
            & {}
        \end{tikzcd};\]
    \end{enumerate}
    subject to the following conditions:
    \begin{enumerate}
        \item[a.] For any objects $x,y,z$ in $\mathfrak{C}$, the equation holds
    \end{enumerate}

    \begin{equation}\label{eqn:Monoidal2NatTransAssociativity}
    \begin{tabular}{@{}ccc@{}}
    
    \stringdiagramfigure{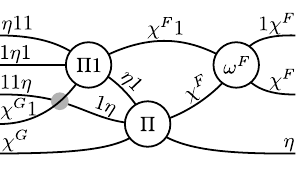} & $=$ &
    \stringdiagramfigure{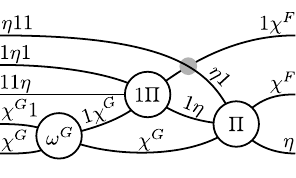} 
        
    \end{tabular}
    \end{equation}

    \begin{enumerate}
        \item[] in $\Hom_\mathfrak{D}(F(x) \, \Box^\mathfrak{D} \, F(y) \, \Box^\mathfrak{D} \, F(z),G(x \, \Box^\mathfrak{C} \, y \, \Box^\mathfrak{C} \, z))$;
    \end{enumerate}
        
    \begin{enumerate}
        \item[b.] For any object $x$ in $\mathfrak{C}$, the equation holds
    \end{enumerate}

    \begin{equation}\label{eqn:Monoidal2NatTransUnitality1}
    \begin{tabular}{@{}ccc@{}}
    
    \stringdiagramfigure{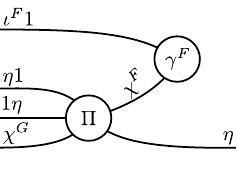} & $=$ &
    \stringdiagramfigure{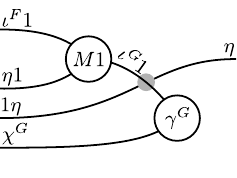} 
        
    \end{tabular}
    \end{equation}

    \begin{enumerate}
        \item[] in $\Hom_\mathfrak{D}(F(x),G(x))$;
    \end{enumerate}
    
    \begin{enumerate}
        \item[c.] For any object $x$ in $\mathfrak{C}$, the equation holds
    \end{enumerate}

    \begin{equation}\label{eqn:Monoidal2NatTransUnitality2}
    \begin{tabular}{@{}ccc@{}}
    
    \stringdiagramfigure{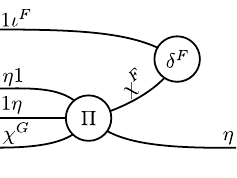} & $=$ &
    \stringdiagramfigure{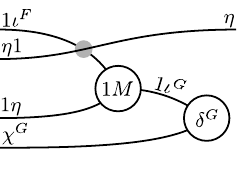} 
        
    \end{tabular}
    \end{equation}

    \begin{enumerate}
        \item[] in $\Hom_\mathfrak{D}(F(x),G(x))$.
    \end{enumerate}
\end{definition}

\begin{notation} \label{not:lax_monoidal_2NatTrans}
    In the above definition, if we do not require the modifications $\Pi$ and $M$ to be invertible, then it is called a \textit{lax} monoidal 2-natural transformation, if the direction of $\Pi$ is the same as given but the direction of $M$ is reversed, or an \textit{oplax} monoidal 2-natural transformation, if the direction of $\Pi$ is reversed but the direction of $M$ is the same as given. The monoidal 2-natural transformations we defined above are called \textit{strong} monoidal 2-natural transformations.
\end{notation}

\begin{remark}\label{rmk:various_lax_conventions}
    Notice that the prefix \textit{``(op)lax"} in:
    \begin{itemize}
        \item (op)lax 2-functors,
        
        \item (op)lax 2-natural transformations,
        
        \item (op)lax monoidal 2-functors,
        
        \item (op)lax monoidal 2-natural transformations,
    \end{itemize} all has different meanings, because we allow non-invertibility on different categorical levels. So there are at least $2^4 = 16$ different combinations of monoidal 2-functors and monoidal 2-natural transformations, depending on where we turn on the invertibility. Even worse, the choices of directions for various (op)lax notions do not always agree in the literature. Some authors use \textit{lax and colax} instead of \textit{lax and oplax}. This situation unfortunately makes the notations messy and confusing. We will follow our conventions explicitly throughout this article.
\end{remark}

\subsection{Monoidal modifications}\label{subsec:monoidal_modification}

\begin{definition} \label{def:monoidal_modification}
    Let $\eta$ and $\phi$ be two monoidal 2-natural transformations from $F$ to $G$, which are monoidal 2-functors given from $\mathfrak{C}$ to $\mathfrak{D}$. A \textit{monoidal modification} from $\eta$ to $\phi$ consists of an underlying modification $\theta:\eta \to \phi$ subject to the two equations:
    \begin{enumerate}
        \item[a.] For any objects $x,y$ in $\mathfrak{C}$, the equation holds
    \end{enumerate}

    \begin{equation}\label{eqn:MonoidalModificationAssociativity}
    \begin{tabular}{@{}ccc@{}}
    
    \stringdiagramfigure{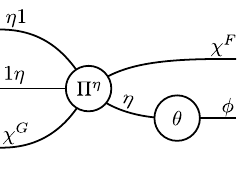} & $=$ &
    \stringdiagramfigure{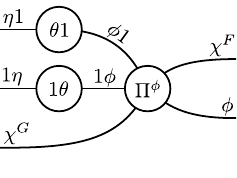} 
        
    \end{tabular}
    \end{equation}

    \begin{enumerate}
        \item[] in $\Hom_\mathfrak{D}(F(x) \, \Box^\mathfrak{D} \, F(y),G(x \, \Box^\mathfrak{C} \, y))$;
    \end{enumerate}

    \begin{enumerate}
        \item[b.] The following equation holds
    \end{enumerate}

    \begin{equation}\label{eqn:MonoidalModificationUnitality}
    \begin{tabular}{@{}ccc@{}}
    
    \stringdiagramfigure{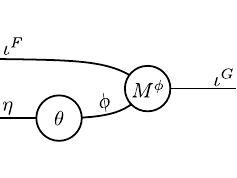} & $=$ &
    \stringdiagramfigure{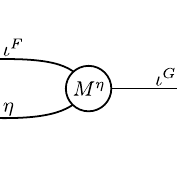} 
        
    \end{tabular}
    \end{equation}

    \begin{enumerate}
        \item[] in $\Hom_\mathfrak{D}(\mathbf{I}^\mathfrak{D},G(\mathbf{I}^\mathfrak{C}))$.
    \end{enumerate}
\end{definition}

\begin{construction}\label{cstr:3-category_of_monoidal_2-categories}
    One has a 3-category $\mathbf{M2Cat}$ consisting of:
    \begin{itemize}[nosep]
        \item Monoidal 2-categories as objects,

        \item Monoidal 2-functors as 1-morphisms,

        \item Monoidal 2-natural transformations as 2-morphisms,

        \item Monoidal modifications as 3-morphisms.
    \end{itemize}

    It has two variants, which we will use later:
    \begin{itemize}[nosep]
        \item $\mathbf{M2Cat}^{lax}$, where we allow 1-morphisms to be lax monoidal 2-functors, but still require 2-morphisms to be strong monoidal 2-natural transformations and 3-morphisms to be monoidal modifications;

        \item $\mathbf{M2Cat}^{oplax}$, where we allow 1-morphisms to be oplax monoidal 2-functors, but still require 2-morphisms to be strong monoidal 2-natural transformations and 3-morphisms to be monoidal modifications.
    \end{itemize}
\end{construction}

\subsection{2-adjunctions}\label{subsec:2-adjunction}

\begin{definition} \label{def:2-adjunction}
    A \textit{2-adjunction} between 2-categories $\mathfrak{C}$, $\mathfrak{D}$ consists of:
    \begin{itemize}
        \item 2-functors $F: \mathfrak{C} \to \mathfrak{D}$, $G: \mathfrak{D} \to \mathfrak{C}$;
        
        \item 2-natural transformation, called \textit{unit}, $\eta: \mathrm{Id}_\mathfrak{C} \to G \circ F$;
        
        \item 2-natural transformation, called \textit{counit}, $\epsilon: F \circ G \to \mathrm{Id}_\mathfrak{D}$;
        
        \item Invertible modifications, called \textit{triangulator}, categorifying the zigzag equations: \[\Phi: (\epsilon \circ_1 \mathbf{1}_F) \circ_2 (\mathbf{1}_F \circ_1 \eta) \to \mathbf{1}_F,\] \[\Psi: (\mathbf{1}_G \circ_1 \epsilon) \circ_2 (\eta \circ_1 \mathbf{1}_G) \to \mathbf{1}_G;\]
    \end{itemize} satisfying the following \textit{swallowtail} equations.

    \begin{enumerate}
        \item[a.] The following equation holds in $\Hom_{\mathbf{Fun}(\mathfrak{C},\mathfrak{D})}(F \circ G,\mathrm{Id}_\mathfrak{D})$:
    \end{enumerate}

    \begin{equation}\label{eqn:swallowtail1}
    \begin{tabular}{@{}cccc@{}}

    \stringdiagramfigure{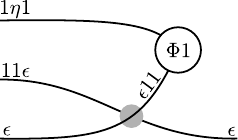} & $=$ &
    \stringdiagramfigure{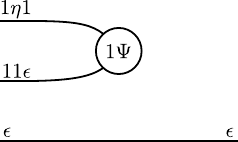} & ;
        
    \end{tabular}
    \end{equation}

    \begin{enumerate}
        \item[b.] The following equation holds in $\Hom_{\mathbf{Fun}(\mathfrak{C},\mathfrak{D})}(\mathrm{Id}_\mathfrak{C}, G \circ F)$:
    \end{enumerate}

    \begin{equation}\label{eqn:swallowtail2}
    \begin{tabular}{@{}cccc@{}}

    \stringdiagramfigure{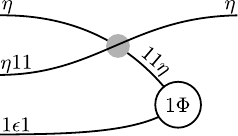} & $=$ &
    \stringdiagramfigure{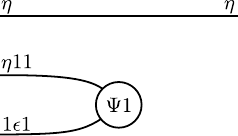} & .
        
    \end{tabular}
    \end{equation}
    
    We say $F$ is the \emph{left adjoint} of $G$ or $G$ is the \emph{right adjoint} of $F$. The 2-adjunction is \textit{strict} if $\Phi$ and $\Psi$ are identities.
\end{definition}

By a similar argument as in the 1-categorical case, one can show that data of the 2-adjunction is \emph{universal}.

\begin{lemma}\label{lem:2-adjunction_is_universal}
    For a given 2-functor $F:\mathfrak{C} \to \mathfrak{D}$, any 2-adjunction with $F$ as the left adjoint is uniquely determined. Dually, for a given 2-functor $G:\mathfrak{D} \to \mathfrak{C}$, any 2-adjunction with $G$ as the right adjoint is uniquely determined.
\end{lemma}

\subsection{Monoidal 2-adjunctions}\label{subsec:monoidal_2-adjunction}
Let us introduce the general construction of monoidal 2-adjunctions first.

\begin{construction}\label{cstr:lax_monoidal_structure_on_the_right_adjoint}
    Let $F:\mathfrak{C} \to \mathfrak{D}$ be an oplax monoidal 2-functor, with a 2-adjunction $(F,G,\eta,\epsilon,\Phi,\Psi)$ where $F$ is the left adjoint. Then there is a canonical lax monoidal structure on $G$:
    \begin{enumerate}
        \item For objects $x$ and $y$ in $\mathfrak{D}$, a 1-morphism $\chi^G_{x,y}:G(x) \, \Box^\mathfrak{C} \, G(y) \to G(x \, \Box^\mathfrak{D} \, y)$ is given by the following composition:
        \begin{equation}\label{eqn:lax_monoidal_structure_on_right_adjoint}
            G(x) \boxt^\mathfrak{C} G(y) \xrightarrow{\eta_{Gx \boxt Gy}} GF(G(x) \boxt^\mathfrak{C} G(y)) \xrightarrow{G(\widehat{\chi}^F_{Gx,Gy})} G(FG(x) \boxt^\mathfrak{D} FG(y)) \xrightarrow{G(\epsilon_x \boxt \epsilon_y)} G(x \boxt^\mathfrak{D} y)
        \end{equation}
        where $\widehat{\chi}^F$ is the oplax monoidal product on $F$;

        \item A 1-morphism $\iota^G:\mathbf{I}^\mathfrak{C} \to G(\mathbf{I}^\mathfrak{D})$ is given by the following composition:
        \begin{equation}
            \mathbf{I}^\mathfrak{C} \xrightarrow{\eta_\mathbf{I}} GF(\mathbf{I}^\mathfrak{C}) \xrightarrow{G(\widehat{\iota}^F)} G(\mathbf{I}^\mathfrak{D})
        \end{equation}
        where $\widehat{\iota}^F$ is the oplax monoidal unit on $F$;

        \item For objects $x,y,z$ in $\mathfrak{D}$, the invertible modification $\omega^G_{x,y,z}$ is given by the following composition:
        \begin{equation}
            \stringdiagramfigure{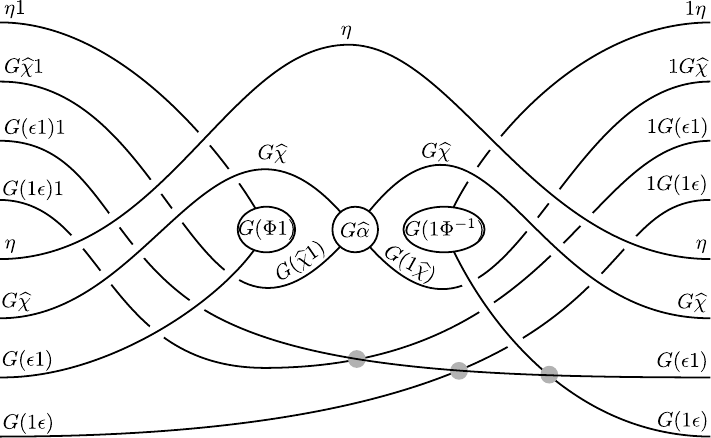}
        \end{equation}
        where $\widehat{\alpha}$ is the associator of the oplax monoidal structure on $F$;

        \item For object $x$ in $\mathfrak{D}$, the invertible modification $\lambda^G_x$ is given by the following composition:
        \begin{equation}
            \stringdiagramfigure{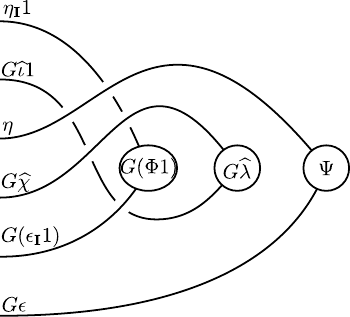}
        \end{equation}
        where $\widehat{\lambda}$ is the left unitor of the oplax monoidal structure on $F$;

        \item For object $x$ in $\mathfrak{D}$, the invertible modification $\rho^G_x$ is given by the following composition:
        \begin{equation}
            \stringdiagramfigure{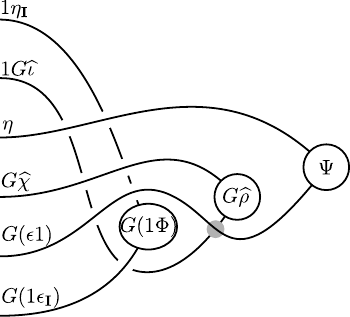}
        \end{equation}
        where $\widehat{\rho}$ is the right unitor of the oplax monoidal structure on $F$.
    \end{enumerate}

    Furthermore, there are compatible data for unit and counit with respect to the monoidal structures on $F$ and $G$:
    \begin{enumerate}
        \item There is an invertible modification with the component
        \begin{equation}\label{eqn:unit_monoidal_structure}
            \stringdiagramfigure{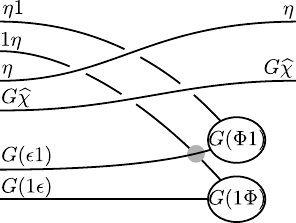}
        \end{equation}
        in $\Hom_\mathfrak{C}(x \, \Box^\mathfrak{C} \, y,G(F x \, \Box^\mathfrak{C} \, F y))$ for objects $x$ and $y$ in $\mathfrak{C}$;

        \item There is an invertible modification with the component
        \begin{equation}\label{eqn:counit_monoidal_structure}
            \stringdiagramfigure{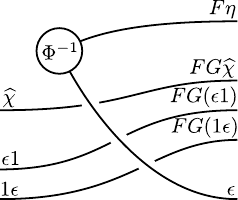}
        \end{equation}
        in $\Hom_\mathfrak{D}(F(G x \, \Box^\mathfrak{D} \, G y),x \, \Box^\mathfrak{D} \, y)$ for objects $x$ and $y$ in $\mathfrak{D}$.
    \end{enumerate}
\end{construction}

Conversely, if $G$ is a lax monoidal 2-functor, then there is a canonical oplax monoidal structure on its left adjoint $F$.

\begin{corollary}\label{cor:monoidal_2-adjunction}
    The left adjoint of a strongly monoidal 2-functor is oplax monoidal. Furthermore, the unit and counit of the 2-adjunction are equipped with canonical monoidal 2-natural transformation structures by inverting the lax monoidal 2-functor structure on the right adjoint in \eqref{eqn:unit_monoidal_structure} and \eqref{eqn:counit_monoidal_structure}. This 2-adjunction is internal to the 3-category of monoidal 2-categories, oplax monoidal 2-functors, monoidal 2-natural transformations, and monoidal modifications.

    Similarly, the right adjoint of a strongly monoidal 2-functor is lax monoidal. The unit and counit of the 2-adjunction are equipped with canonical monoidal 2-natural transformation structures by inverting the oplax monoidal 2-functor structure on the left adjoint in \eqref{eqn:unit_monoidal_structure} and \eqref{eqn:counit_monoidal_structure}. This 2-adjunction is internal to the 3-category of monoidal 2-categories, lax monoidal 2-functors, monoidal 2-natural transformations, and monoidal modifications.
\end{corollary}
\section{Left and Right Centers}

Let $\mathfrak{B}$ be a braided monoidal 2-category and let $B$ be an algebra in $\mathfrak{B}$.

\subsection{Commuting objects}
The 2-category $\Comm(B)$ of \emph{left commuting objects over $B$} is defined as follows.

\begin{definition}\label{def:commuting_objects}
A left commuting object over $B$ is a triple $(x,f,\psi^f)$, where
\begin{enumerate}[nosep, label=(\roman*)]
    \item $x$ is an object of $\mathfrak{B}$;
    
    \item $f:x\to B$ is a 1-morphism;
    
    \item $\psi^f$ is a 2-isomorphism: 
    \begin{equation}
        \psi^f \colon m \circ (f\boxt 1_B) \to m \circ (1_B\boxt f) \circ b_{x,B},
    \end{equation}
    subject to the following coherence conditions.
\end{enumerate}

The compatibility of $\psi^f$ with the associator of $B$ is the following equality in $\mathbf{Hom}_\mathfrak{B}(x \, \Box \, B \, \Box \, B,B)$:
\begin{equation}\label{eqn:commuting_object_associativity}
    \begin{array}{@{}c@{\quad}c@{\quad}c@{}}
        \stringdiagramfigure{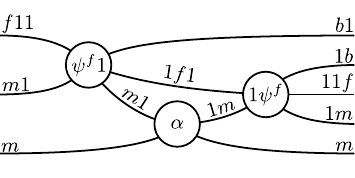}
        &
        =
        &
        \stringdiagramfigure{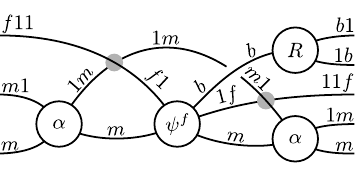} \quad .
    \end{array}
\end{equation}

The compatibility of $\psi^f$ with the unitors of $B$ is the following equality in $\mathbf{Hom}_\mathfrak{B}(x,B)$:

\begin{equation}\label{eqn:commuting_object_unitality}
    \begin{array}{@{}c@{\quad}c@{\quad}c@{}}
        \stringdiagramfigure{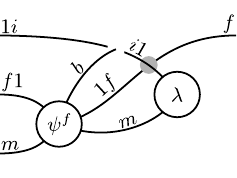}
        &
        =
        &
        \stringdiagramfigure{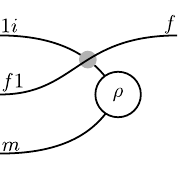} \quad .
    \end{array}
\end{equation}
\end{definition}

\begin{definition}\label{def:commuting_objects_morphisms}
A 1-morphism of commuting objects $(x,f,\psi^f)\to (y,g,\psi^g)$ consists of:
\begin{enumerate}[nosep, label=(\roman*)]
    \item A 1-morphism $h \colon x\to y$ in $\mathfrak{B}$;
    
    \item A 2-isomorphism $\chi^h \colon f \to g \circ h$ subject to the following equation in $\mathbf{Hom}_\mathfrak{B}(x \, \Box \, B,B)$:
\end{enumerate}

\begin{equation}\label{eqn:commuting_object_1morphism}
    \begin{array}{@{}c@{\quad}c@{\quad}c@{}}
        \stringdiagramfigure{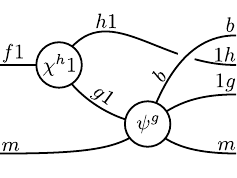}
        &
        =
        &
        \stringdiagramfigure{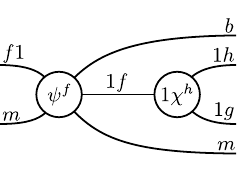} \quad .
    \end{array}
\end{equation}
    
A 2-morphism from $(h,\chi^h)$ to $(k,\chi^k)$, which are 1-morphisms from $(x,f,\psi^f)$ to $(y,g,\psi^g)$, consists of a 2-morphism $\xi:h \to k$ in $\mathfrak{B}$, subject to the following equation in $\mathbf{Hom}_\mathfrak{B}(x,B)$:
\begin{equation}\label{eqn:commuting_object_2morphism}
    \begin{array}{@{}c@{\quad}c@{\quad}c@{}}
        \stringdiagramfigure{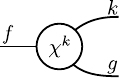}
        &
        =
        &
        \stringdiagramfigure{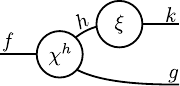} \quad .
    \end{array}
\end{equation}
\end{definition}

\begin{definition}\label{def:left_center}
The \emph{left center} of $B$ is the terminal object of $\Comm(B)$, when it exists. We denote it by $\Zl(B)$, and write
\[
    \zeta:\Zl(B)\to B
\]
for the canonical 1-morphism.
\end{definition}

\begin{definition}\label{def:right_center}
The \emph{right center} $\Zr(B)$ is defined by the same universal property, but with the reverse braiding used in Definitions \ref{def:commuting_objects} and \ref{def:commuting_objects_morphisms}.
\end{definition}

\subsection{Braided algebra structure}
Let us construct the braided algebra structure on the left center $\Zl(B)$, assuming it exists. The construction for $\Zr(B)$ is analogous.

\begin{proposition}\label{prop:left_center_braided}
The left center $\mathrm{Z}^l_1(B)$ is a braided algebra in $\mathfrak{B}$. Moreover, the canonical 1-morphism $\zeta: \mathrm{Z}^l_1(B) \to B$ is an algebra 1-morphism in $\mathfrak{B}$.
\end{proposition}

\begin{proof}
The braided algebra structure on $\mathrm{Z}^l_1(B)$ is given by the following:
    \begin{itemize}
        \item The unit is induced by considering $i:\mathbf{I} \to B$ as a commuting object over $B$ with the 2-isomorphism 
        \[\begin{tikzcd}
            {\mathbf{I} \, \Box \, B}
                \arrow[dd,"b_{\mathbf{I},B}"']
                \arrow[r,"i 1" {name=U}]
            & {B \, \Box \, B}
                \arrow[rd,"m"]
                \arrow[d,Rightarrow,shorten <=5pt,shorten >=5pt,"{\lambda}"]
            & {}
            \\ {}
            & {B}
                \arrow[r,equal]
                \arrow[lu,equal]
                \arrow[ld,equal]
                \arrow[d,Rightarrow,shorten <=5pt,shorten >=5pt,"{\rho^{-1}}"]
            & {B}
            \\{B \, \Box \, \mathbf{I}}
                \arrow[r,"1 i"' {name=D}]
            & {B \, \Box \, B}
                \arrow[ur,"m"']
            & {}
        \end{tikzcd}. \] 
        
        \noindent Notice that $b_{\mathbf{I},B}$ is assumed to be the identity of $B$. By the universal property of the left center, we obtain a 1-morphism in $\mathbf{Comm}(B)$: \[\widetilde{i}:\mathbf{I} \to \mathrm{Z}^l_1(B)\] with a 2-isomorphism 
        \[\begin{tikzcd}
            {\mathbf{I}}
                \arrow[rr,"i"]
                \arrow[dr,"\widetilde{i}"']
            & {}
            & {B}
            \\ {}
            & {\mathrm{Z}^l_1(B)}
                \arrow[ur,"\zeta"']
                \arrow[u,Rightarrow,shorten <=5pt,shorten >=2pt,"{\chi^{\widetilde{i}}}"]
            & {}
        \end{tikzcd}. \]
        
        \item The multiplication is induced by \[ \mathrm{Z}^l_1(B) \, \Box \,  \mathrm{Z}^l_1(B) \xrightarrow{\zeta \, \Box \, \zeta} B \, \Box \, B \xrightarrow{m} B\] as a commuting object over $B$, with the 2-isomorphism
        \begin{center}
            \stringdiagramfigure{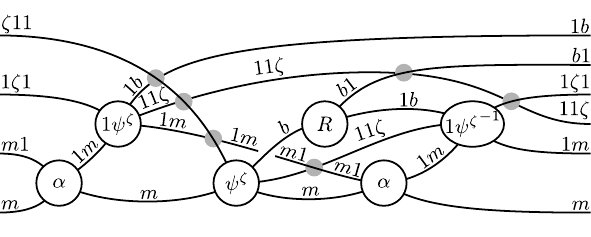} \quad .
        \end{center}

        \noindent By the universal property of the left center, we obtain a 1-morphism in $\mathbf{Comm}(B)$: \[\widetilde{m}:\mathrm{Z}^l_1(B) \, \Box \,  \mathrm{Z}^l_1(B) \to \mathrm{Z}^l_1(B)\] with a 2-isomorphism 
        \[\begin{tikzcd}
            {\mathrm{Z}^l_1(B) \, \Box \,  \mathrm{Z}^l_1(B)}
                \arrow[r,"\zeta \, \Box \, \zeta" ]
                \arrow[d,"\widetilde{m}"']
            & {B \, \Box \, B}
                \arrow[d,"m"]
            \\ {\mathrm{Z}^l_1(B)}
                \arrow[r,"\zeta"']
                \arrow[ur,Rightarrow,shorten <=15pt,shorten >=15pt,"{\chi^{\widetilde{m}}}"]
            & {B}
        \end{tikzcd}.\]
        
        \item The associator is induced by the universal property of the left center, considering the following 2-morphism in $\mathbf{Comm}(B)$: 
        \begin{center}
            \stringdiagramfigure{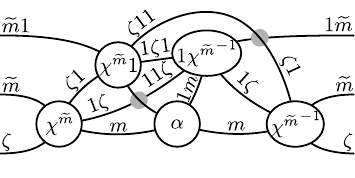} \quad .
        \end{center}
        
        \item The left unitor is induced by the universal property of the left center, considering the following 2-morphism in $\mathbf{Comm}(B)$: 
        \begin{center}
            \stringdiagramfigure{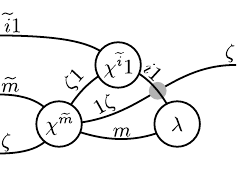} \quad .
        \end{center}
        
        \item The right unitor is induced by the universal property of the left center, considering the following 2-morphism in $\mathbf{Comm}(B)$: 
        \begin{center}
            \stringdiagramfigure{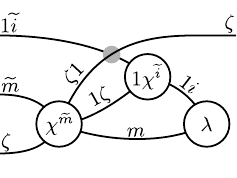} \quad .
        \end{center}
        
        \item The braiding is induced by the universal property of the left center, considering the following 2-morphism in $\mathbf{Comm}(B)$: 
        \begin{center}
            \stringdiagramfigure{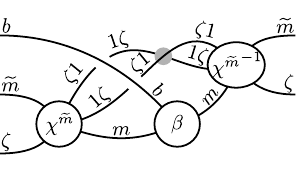} \quad .
        \end{center}
    \end{itemize}

    \noindent The algebra 1-morphism structure on $\zeta$ is given by $\chi^{\widetilde{m}}$ and $\chi^{\widetilde{i}}$ we constructed above.
\end{proof}

\begin{corollary}
If $\Zr(B)$ exists, then $\Zr(B)$ is naturally a braided algebra in $\mathfrak{B}$, and the canonical morphism $\Zr(B)\to B$ is an algebra 1-morphism.
\end{corollary}

\subsection{Morita invariance} Two algebras $A$ and $B$ in the braided monoidal 2-category $\mathfrak{B}$ is \emph{Morita equivalent} if there exists an $(A,B)$-bimodule $M$ and a $(B,A)$-bimodule $N$ such that $M \boxt_B \, N \simeq A$ as $(A,A)$-bimodules and $N \, \boxt_A \, M \simeq B$ as $(B,B)$-bimodules. Equivalently, $A$ and $B$ are Morita equivalent if their right module 2-categories $\Mod_\mathfrak{B}(A)$ and $\Mod_\mathfrak{B}(B)$ are equivalent as left $\mathfrak{B}$-module 2-categories.

\begin{construction}\label{cstr:opposite_algebra}
    For any algebra $(A,m,i,\alpha,\lambda,\rho)$ in $\mathfrak{B}$, we can construct its \emph{opposite algebra}:
    \begin{itemize}[nosep]
        \item The underlying object is $A$;
        
        \item The multiplication $m^\mathrm{op}$ is given by 
        \begin{equation}
            A \boxt A \xrightarrow{b_{A,A}} A \boxt A \xrightarrow{m} A;
        \end{equation}

        \item The unit is the same $i \colon \mathbf{I} \to A$;
        
        \item The associator $\alpha^\mathrm{op}$ is given by the following 2-isomorphism:
        \begin{equation}
            \stringdiagramfigure{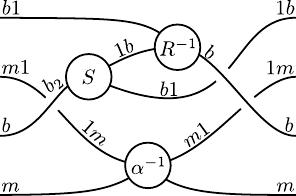} \quad ;
        \end{equation}

        \item The left unitor $\lambda^\mathrm{op}$ is given by the following 2-isomorphism:
        \begin{equation}
            \stringdiagramfigure{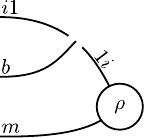} \quad ;
        \end{equation} 

        \item The right unitor $\rho^\mathrm{op}$ is given by the following 2-isomorphism:
        \begin{equation}
            \stringdiagramfigure{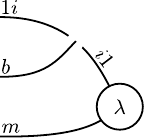} \quad .
        \end{equation}
    \end{itemize}

    We denote the opposite algebra of $A$ by $A^\mathrm{op}$. It is straightforward to check that $A^\mathrm{op}$ is indeed an algebra in $\mathfrak{B}$.
\end{construction}

\begin{construction}\label{cstr:algebra_on_product}
    Suppose $A$ and $B$ are two algebras in the braided monoidal 2-category $\mathfrak{B}$. The monoidal product $A \boxt B$ carries an algebra structure in $\mathfrak{B}$, given as follows:
    \begin{itemize}[nosep]
        \item The multiplication $m^{A\boxt B}$ is given by 
        \begin{equation}
            A \boxt B \boxt A \boxt B \xrightarrow{1_A \boxt b_{B,A} \boxt 1_B} A \boxt A \boxt B \boxt B \xrightarrow{m^A\boxt m^B} A \boxt B ;
        \end{equation}
        
        \item The unit $i^{A\boxt B}$ is given by $\mathbf{I} = \mathbf{I} \boxt \mathbf{I} \xrightarrow{i^A \boxt i^B} A \boxt B$;

        \item The associator $\alpha^{A\boxt B}$ is given by the following 2-isomorphism:
        \begin{equation}
            \stringdiagramfigure{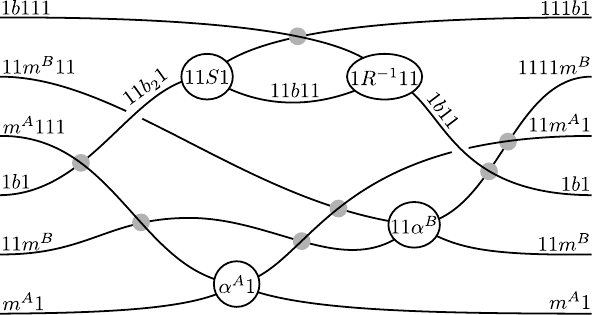} \quad ;
        \end{equation}

        \item The left unitor $\lambda^{A\boxt B}$ is given by the following 2-isomorphism:
        \begin{equation}
            \stringdiagramfigure{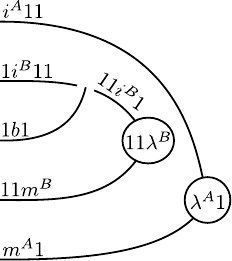} \quad ;
        \end{equation}

        \item The right unitor $\rho^{A\boxt B}$ is given by the following 2-isomorphism:
        \begin{equation}
            \stringdiagramfigure{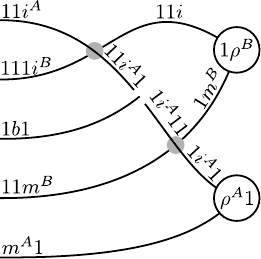} \quad .
        \end{equation}
    \end{itemize}
\end{construction}

\begin{construction}\label{cstr:BimoduleAsRightModuleOverProductAlgebra}
    Let $A$ be an algebra in the braided monoidal 2-category $\mathfrak{B}$. There exist non-canonical equivalences between $\Bimod_\mathfrak{B}(A)$ and $\Mod_\mathfrak{B}(A^\mathrm{op} \boxt A)$. We fix one such equivalence constructed as follows. 
    
    Given any $(A,A)$-bimodule $(M,l^M,r^M,\mu^M,\nu^M,\lambda^M,\rho^M,\alpha^M)$ in $\mathfrak{B}$, we endow the same underlying object $M$ with a right $A^\mathrm{op} \boxt A$-action via
    \begin{equation}
        M\boxt A\boxt A \xrightarrow{b_{M,A}\boxt 1_{A}} A\boxt M\boxt A \xrightarrow{l^M\boxt 1_{A}} M\boxt A \xrightarrow{r^M} M,
    \end{equation}
    together with the associator and right unitor:
    \begin{equation}
        \widetilde{\nu}^{M} := \quad \stringdiagramfigure{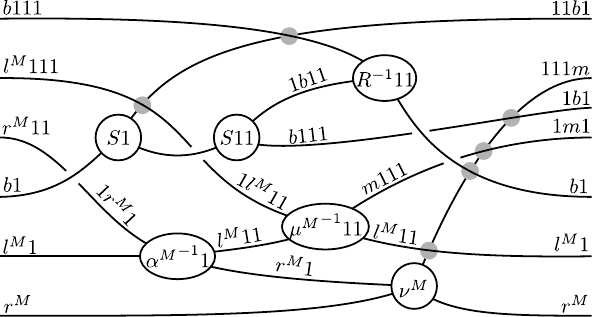} \quad ,
    \end{equation}
    \begin{equation}
        \widetilde{\rho}^{M} := \quad \stringdiagramfigure{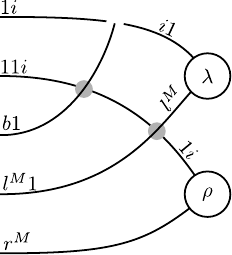} \quad .
    \end{equation}
\end{construction}

\begin{theorem}\label{thm:left_center_as_internal_hom}
    The left center $\Zl(A)$ is equivalent to the internal endo-Hom of the unit $A$ in the 2-category $\Mod_\mathfrak{B}(A^\mathrm{op} \boxt A)$.
\end{theorem}

\begin{proof}
    For a left commuting object $(x,f,\psi^f)$ over $A$, consider the 1-morphism $x \boxt A \xrightarrow{f \boxt 1_A} A \boxt A \xrightarrow{m} A$ with the right $A^\mathrm{op} \boxt A$-module 1-morphism structure given by
    \begin{equation}
        \begin{array}{@{}c@{\quad}c@{}}
        \stringdiagramfigure{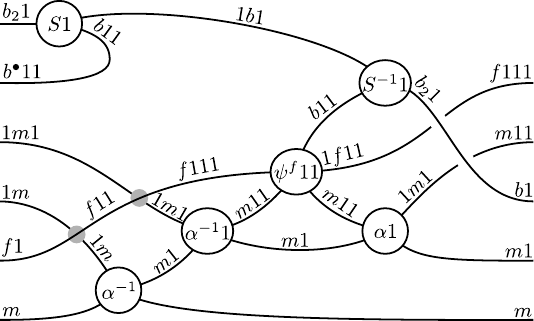} & .
        \end{array}
    \end{equation}

    Conversely, for a right $A^\mathrm{op} \boxt A$-module 1-morphism $g \colon x \boxt A \to A$, we can construct a 1-morphism $x = x \boxt \mathbf{I} \xrightarrow{1_x \boxt i} x \boxt A \xrightarrow{g} A$ with a 2-isomorphism $\psi^g$
    \begin{equation}
        \begin{array}{@{}c@{\quad}c@{}}
        \stringdiagramfigure{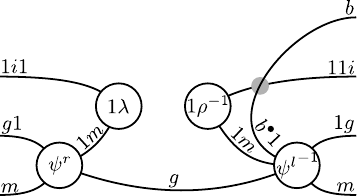} & ,
        \end{array}
    \end{equation}
    promoting this 1-morphism to a left commuting object over $A$, where $\psi^l$ is the left $A$-module 1-morphism structure of $g$, and $\psi^r$ is the right $A$-module 1-morphism structure of $g$, defined as
    \begin{equation}
        \stringdiagramfigure{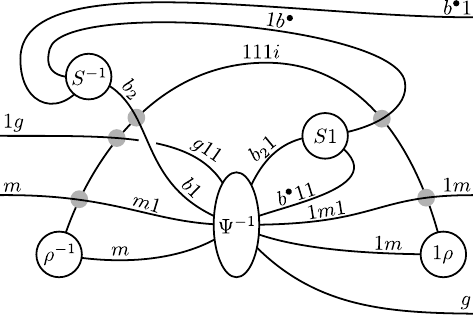}
    \end{equation}
    and 
    \begin{equation}
        \stringdiagramfigure{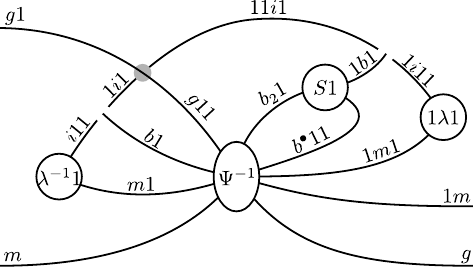}
    \end{equation}
    respectively. Here $\Psi$ is the right $A^\mathrm{op} \boxt A$-module 1-morphism structure of $g$. It is straightforward to check that these two constructions are mutually inverse, and that they preserve the 1-morphism and 2-morphism structures. This produces an equivalence between $\Zl(A)$ and the internal endo-Hom of the unit $A$ in the 2-category $\Mod_\mathfrak{B}(A^\mathrm{op} \boxt A)$.
\end{proof}

\begin{remark}
    If we consider an alternative equivalence in Construction \ref{cstr:BimoduleAsRightModuleOverProductAlgebra} by replacing the braiding $b_{A,M}$ by the reversed braiding $b^\bullet_{M,A}$, then we can similarly show that the right center $\Zr(A)$ is equivalent to the internal endo-Hom of the unit $A$ in the 2-category $\Mod_\mathfrak{B}(A^\mathrm{op} \boxt A)$.
\end{remark}

\begin{corollary}\label{cor:left_center_is_Morita_invariant}
    The left center $\Zl(A)$ is Morita invariant. That is, if $A$ and $B$ are Morita equivalent algebras in $\mathfrak{B}$, then $\Zl(A)$ and $\Zl(B)$ are equivalent as braided algebras in $\mathfrak{B}$. Similarly, the right center $\Zr(A)$ is Morita invariant.
\end{corollary}

\section{Base Change for Module 2-categories}
Let $A$ be a braided algebra in $\mathfrak{B}$, and assume that relative tensor products over $A$ exist.

\begin{construction} \label{cstr:induction_2functor}
    There is a 2-functor $\mathrm{Ind}^+:\Mod_\mathfrak{B}(A) \to \Bimod_\mathfrak{B}(A,A)$ defined as follows.

    Given any right $A$-module $(M,r^M,\nu^M,\rho^M)$, we can endow it with a left $A$-action $l^M:A \, \Box \, M \xrightarrow{b_{A,M}} M \, \Box \, A \xrightarrow{r^M} M$, associator $\mu^M$, and left unitor $\lambda^M$:
    \[
    \mu^M := \quad \stringdiagramfigure{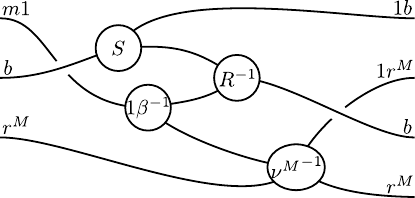} \quad ,
    \]
    \[
    \lambda^M := \quad \stringdiagramfigure{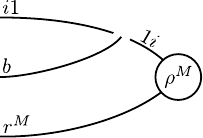} \quad .
    \]

    \noindent The left and right $A$-actions on $M$ are compatible with the 2-isomorphism:
    \[
    \alpha^M := \quad \stringdiagramfigure{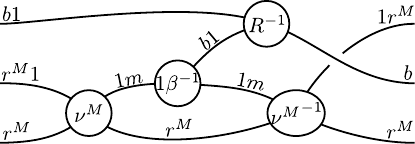} \quad .
    \]

    \noindent Given a right $A$-module 1-morphism $(f, \phi^f):M \to N$, we induce a left $A$-module 1-morphism structure:
    \[
    \chi^f := \quad \stringdiagramfigure{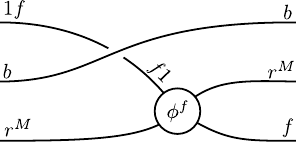} \quad .
    \]

    \noindent Finally, it is routine to check that the above structures satisfy the corresponding coherence conditions. This produces a 2-functor \[\mathrm{Ind}^+:\Mod_\mathfrak{B}(A) \to \Bimod_\mathfrak{B}(A,A).\]
\end{construction}

\begin{remark}\label{rmk:induction_2functor_is_fully_faithful}
    It is straightforward to see that this induction 2-functor is fully faithful on 2-morphisms.
\end{remark}

\begin{notation}\label{not:induced_monoidal_product_on_right_modules}
    The induction 2-functor $\mathrm{Ind}^+$ endows $\Mod_\mathfrak{B}(A)$ with a monoidal structure given by the relative tensor product, which we denote by $\Box^+_A$.

    Meanwhile, there is another induced $B$-action on $M$ by replacing the braiding $(b,R,S)$ by its reversal $(b^\bullet,S^\bullet,R^\bullet)$: 
    \[B \, \Box \, M \xrightarrow{b^\bullet_{M,B}} M \, \Box \, B \xrightarrow{n^M} M.\]
    Analogously we can define the other coherence data for this $B$-action. This produces another induction 2-functor $\mathrm{Ind}^-:\Mod_\mathfrak{B}(A) \to \Bimod_\mathfrak{B}(A,A)$, which is monoidal with respect to the relative tensor product $\Box^-_A$ induced by the reversed braiding.

    For consistency, when we refer to the monoidal structure on $\mathbf{Mod}_\mathfrak{B}(B)$, unless otherwise specified, we always refer to the monoidal structure $\Box^+_B$. In this case, we will simply write $\Box_B$ for $\Box^+_B$.
\end{notation}

\begin{proposition}\label{prop:algebras-in-modules}
An algebra in $\Mod_{\mathfrak{B}}(A)$ is equivalent to an algebra $B$ in $\mathfrak{B}$ together with a braided algebra 1-morphism
\[
    \widetilde{f}:A\to \Zl(B).
\]
\end{proposition}

\begin{proof}
    The forgetful 2-functor \[\mathbf{U}: \mathbf{Mod}_\mathfrak{B}(A) \to \mathfrak{B}\] is lax monoidal. The image of the unit of an algebra in $\mathbf{Mod}_\mathfrak{B}(A)$ under this forgetful 2-functor becomes an algebra 1-morphism $f:A \to B$ in $\mathfrak{B}$. Moreover, $f$ can be promoted to a commuting object over $B$ with the 2-isomorphism induced by 
    \[
    \begin{tikzcd}
        {A \, \Box \, B}
            \arrow[dd,"{b}_{A,B}"']
            \arrow[r,"t_{A,B}"]
        & {A \, \Box_A \, B}
            \arrow[r,"f \, \Box_A \, B" {name=U}]
            \arrow[dr,"\pmb{l}_B"']
        & {B \, \Box_A \, B}
            \arrow[rd,"m"]
            \arrow[d,Rightarrow,shorten <=5pt,shorten >=5pt,"{\lambda}"]
        & {}
        \\ {}
        & {}
        & {B}
            \arrow[r,equal]
            \arrow[d,Rightarrow,shorten <=5pt,shorten >=5pt,"{\rho^{-1}}"]
        & {B}
        \\ {B \, \Box \, A}
            \arrow[r,"t_{B,A}"']
        & {B \, \Box_A \, A}
            \arrow[r,"B \, \Box_A \, f"' {name=D}]
            \arrow[ur,"\pmb{r}_B"]
        & {B \, \Box_A \, B}
            \arrow[ur,"m"']
        & {}
    \end{tikzcd},
    \]
    where 
    \begin{itemize}
        \item $\pmb{l}$ and $\pmb{r}$ are the left and right 1-unitors of $\mathbf{Mod}_\mathfrak{B}(A)$ in \eqref{eqn:relative_tensor_product_unitor};
        
        \item $t_{A,B}$ and $t_{B,A}$ are the canonical $A$-balanced 1-morphisms in \eqref{eqn:relative_tensor_product_balanced_1morphism}.
    \end{itemize} 
    
    \noindent The pentagon on the left is filled by a 2-isomorphism since the left $A$-action on $B$ is induced by the right $A$-action on $B$ in Construction \ref{cstr:induction_2functor}. Hence, by the universal property of the left center, $f$ factors through a 1-morphism in $\mathbf{Comm}(B)$: \[\widetilde{f}:A \to \mathrm{Z}^l_1(B).\] 

    \noindent The algebra 1-morphism structure of $\widetilde{f}$ is constructed in a way similar to the proof of Proposition \ref{prop:left_center_braided}. Finally, one can check that it is indeed a braided algebra 1-morphism by verifying equation \eqref{eqn:BraidedAlgebra1Morphism}.

    Conversely, given a braided algebra 1-morphism $\widetilde{f}:A \to \mathrm{Z}^l_1(B)$, we can endow $B$ with a right $A$-module structure. Then the algebra structure of $B$ can be promoted to an algebra structure in $\mathbf{Mod}_\mathfrak{B}(A)$ by the universal property of the relative tensor product $\Box_A$.
\end{proof}

\begin{theorem}\label{thm:base-change-ordinary-modules}
Under the hypotheses of Proposition \ref{prop:algebras-in-modules}, there is an equivalence of 2-categories
\[
    \Mod_{\mathfrak{B}}(B)
    \simeq
    \Mod_{\Mod_{\mathfrak{B}}(A)}(B).
\]
\end{theorem}

\begin{proof}
The lax monoidal forgetful 2-functor $\Mod_{\mathfrak{B}}(A)\to \mathfrak{B}$ induces a 2-functor 
\[\Mod_{\Mod_{\mathfrak{B}}(A)}(B) \to \Mod_{\mathfrak{B}}(B).\] 
For the other direction, a right $B$-module $M$ obtains a compatible right $A$-action by precomposing with
\[
    A\xrightarrow{\widetilde{f}}\Zl(B)\xrightarrow{\zeta}B.
\]
The commuting object data for $\widetilde{f}$ enables us to descend the $B$-action on $M$ to the relative tensor product, so $M$ becomes a right $B$-module internal to $\Mod_{\mathfrak{B}}(A)$. Lastly, one can check that the above constructions are mutually inverse up to 2-natural equivalences.
\end{proof}

\section{Full Centers}

The left and right centers of §\ref{def:left_center} live \emph{internal} to a braided monoidal 2-category $\mathfrak{B}$. The full center construction instead places the center of an algebra \emph{relative} to the Drinfeld center.

\subsection{Drinfeld center}
Let us first recall the definition of the Drinfeld center of a monoidal 2-category. The center of a monoidal 2-category was first constructed by Baez--Neuchl \cite{BN96} as a braided monoidal 2-category, in analogy with the Drinfeld center of a monoidal category as its quantum double; the associated generalized and higher centers were further studied by Crans \cite{Cr}. More recently, the center was computed explicitly for the monoidal 2-category $\mathbf{2Vect}_G^\pi$ of $\pi$-twisted $G$-graded finite semisimple categories by Kong--Tian--Zhou \cite{KTZ20}, where it describes the topological defects of the $(3+1)$d Dijkgraaf--Witten theory, and was developed for fusion 2-categories by D\'ecoppet \cite{D9}. Our formulation below, phrased in terms of the Drinfeld centralizer of a monoidal 2-functor, follows the semistrict conventions of D\'ecoppet.

\begin{definition} \label{def:drinfeld_centralizer}
    For a monoidal 2-functor $F:\mathfrak{C} \to \mathfrak{D}$ between semistrict monoidal 2-categories $\mathfrak{C}$ and $\mathfrak{D}$, its \textit{Drinfeld centralizer} is a monoidal 2-category $\mathscr{Z}_1(F)$ where
    \begin{itemize}
        \item [(1)] An object is a triple $(x,b^x_{-},R^x_{-,-})$, where 
        \begin{itemize}
            \item $x$ is an object in $\mathfrak{D}$;

            \item $b^x_{-}$ is a 2-natural isomorphism
            $$\begin{tikzcd}[sep=small]
            {x F(y_0)} 
                \arrow[rrr,"b^x_{y_0}"] 
                \arrow[ddd,"1 F(f)"']
            & {} 
            & {} 
            & {F(y_0) x}
                \arrow[ddd,"F(f) 1"]
                \arrow[dddlll,Rightarrow,"b^x_{f}"',shorten >=3ex, shorten <= 3ex] 
            \\ {} & {} & {} & {}
            \\ {} & {} & {} & {}
            \\ {x F(y_1)}
                \arrow[rrr,"b^x_{y_1}"']
            & {} 
            & {} 
            & {F(y_1) x}
            \end{tikzcd}$$ which is natural in 1-morphism $f:y_0 \to y_1$ in $\mathfrak{C}$;

            \item $R^x_{-,-}$ is an invertible modification 
            $$\begin{tikzcd}[sep=25pt]
            xF(yz) 
                \arrow[rr, "b^x_{yz}"]
            & {} \arrow[d, Rightarrow, "R^x_{y,z}"]          & F(yz)x 
            \\ {x F(y)F(z)}
                \arrow[r, "b^x_{y} 1"'] 
                \arrow[u,"1 F_{y,z}"]
            & {F(y)xF(z) }
                \arrow[r, "1 b^x_{z}"'] 
            & {F(y)F(z)x}    
                \arrow[u,"F_{y,z}1"']
            \end{tikzcd}$$ 
            \end{itemize} satisfying the coherence condition: for objects $y,z,w$ in $\mathfrak{C}$, the equation holds
            
            \begin{equation}\label{eqn:DrinfeldCenterObject}
            \begin{tabular}{@{}ccc@{}}

            \stringdiagramfigure{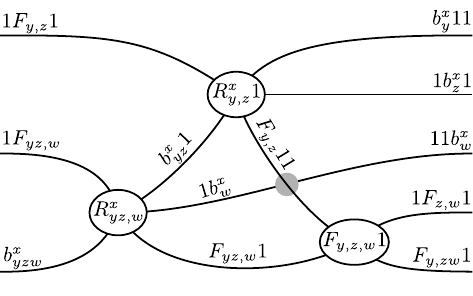} & = & \stringdiagramfigure{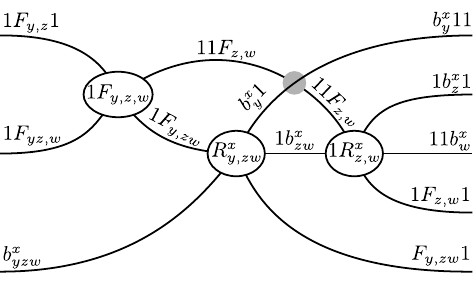}

            \end{tabular}
            \end{equation}
            
        \begin{itemize}
            \item [] in $\mathbf{Hom}_\mathfrak{D}(x \, \Box^\mathfrak{D} \, F(y) \, \Box^\mathfrak{D} \, F(z) \, \Box^\mathfrak{D} \, F(w), F(y \, \Box^\mathfrak{C} \, z \, \Box^\mathfrak{C} \, w) \, \Box^\mathfrak{D} \, x)$.
        \end{itemize}

        \item [(2)] A 1-morphism between objects $(x,b^x_{-},R^x_{-,-})$ and $(y,b^y_{-},R^y_{-,-})$ is a pair $(g,b^g_{-})$ where
        \begin{itemize}
            \item $g:x \to y$ is a 1-morphism in $\mathfrak{D}$;

            \item $b^g_{-}$ is an invertible modification
            $$\begin{tikzcd}[sep=small]
            {x F(z)} 
                \arrow[rrr,"b^x_{z}"] 
                \arrow[ddd,"g 1"']
            & {} 
            & {} 
            & {F(z) x}
                \arrow[ddd,"1 g"]
                \arrow[dddlll,Rightarrow,"b^g_{z}"',shorten >=3ex, shorten <= 3ex] 
            \\ {} & {} & {} & {}
            \\ {} & {} & {} & {}
            \\ {y F(z)}
                \arrow[rrr,"b^y_{z}"']
            & {} 
            & {} 
            & {F(z) y}
            \end{tikzcd}$$ which is natural for object $z$ in $\mathfrak{C}$, satisfying the coherence condition: for objects $z,w$ in $\mathfrak{C}$, the equation holds
            \end{itemize}
            \begin{equation}\label{eqn:DrinfeldCenter1Morphism}
            \begin{tabular}{@{}ccc@{}}

            \stringdiagramfigure{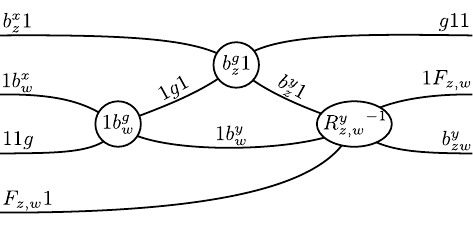} & = & \stringdiagramfigure{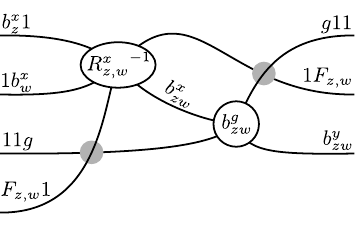}

            \end{tabular}
            \end{equation}
            \begin{itemize}
                \item [] in $\mathbf{Hom}_\mathfrak{D}(x \, \Box^\mathfrak{D} \, F(z) \, \Box^\mathfrak{D} \, F(w), F(z \, \Box^\mathfrak{C} \, w) \, \Box^\mathfrak{D} \, y)$.
            \end{itemize}
        
        \item [(3)] A 2-morphism between $(g,b^g_{-})$ and $(h,b^h_{-})$, which are 1-morphisms between $(x,b^x_{-},R^x_{-,-})$ and $(y,b^y_{-},R^y_{-,-})$, consists of a 2-morphism $\varphi:g \to h$ in $\mathfrak{D}$ satisfying the coherence condition: for object $z$ in $\mathfrak{C}$, the equation holds
            \begin{equation}\label{eqn:DrinfeldCenter2Morphism}
            \begin{tabular}{@{}ccc@{}}

            \stringdiagramfigure{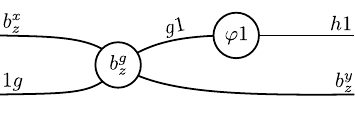} & = & \stringdiagramfigure{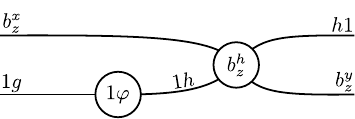}

            \end{tabular}
            \end{equation}
            
            \noindent in $\mathbf{Hom}_\mathfrak{D}(x \, \Box^\mathfrak{D} \, F(z), F(z) \, \Box^\mathfrak{D} \, y)$.
            
        \item [(4)] Monoidal unit is given by the monoidal unit $\mathbf{I}^\mathfrak{D}$ together with $b^\mathbf{I}_{-}$ given by the identity 2-natural isomorphism on $F$, and $R^\mathbf{I}_{-,-}$ given by the identity modification.

        \item [(5)] Monoidal product of $(x,b^x_{-},R^x_{-,-})$ and $(y,b^y_{-},R^y_{-,-})$ is given by the monoidal product $x \, \Box^\mathfrak{D} \, y$ together with $b^{xy}_{-} :=(b^x_{-} \, \Box^\mathfrak{D} \, \mathbf{1}_y) \circ (\mathbf{1}_x \, \Box^\mathfrak{D} \, b^y_{-})$ and $R^{xy}_{-,-}$ given by 
        \[\begin{tikzcd}[column sep=35pt]
            {xyF(zw)}
                \arrow[rr, "b^{xy}_{zw}"]
                \arrow[rd, "1 b^y_{zw}"] 
            & {}
            & {F(zw)xy}
            \\ {xyF(z)F(w)}
                \arrow[dd,"1 b^y_{z}1"']
                \arrow[u,"11F_{z,w}"]
            & {xF(zw)y}
                \arrow[ru, "b^x_{zw}1"] 
            & {F(z)F(w)xy}  
                \arrow[u,"F_{z,w}11"']
            \\ {}
                \arrow[r, Rightarrow, "1 R^y_{z,w}",shorten >= 20pt,shorten <= 20pt]
            & {xF(z)F(w)y} 
                \arrow[rd,"b^x_{z}11"']
                \arrow[r, Rightarrow, "R^x_{z,w} 1",shorten >= 20pt,shorten <= 20pt]
                \arrow[u,"1F_{z,w}1"]
            & {}
            \\ {xF(z)yF(w)}
                \arrow[ru,"11b^y_{w}"']
            & {}
            & {F(z)xF(w)y}
                \arrow[uu,"1b^x_{w}1"']
            \end{tikzcd}\] for objects $z,w$ in $\mathfrak{C}$.

        \item [(6)] Monoidal product of 1-morphisms:
        \begin{itemize}
            \item $(g,b^g_{-})$, from $(x_0,b^{x_0}_{-},R^{x_0}_{-,-})$ to $(y_0,b^{y_0}_{-},R^{y_0}_{-,-})$, 
            
            \item $(h,b^h_{-})$, from $(x_1,b^{x_1}_{-},R^{x_1}_{-,-})$ to $(y_1,b^{y_1}_{-},R^{y_1}_{-,-})$,
        \end{itemize}  
        
        \noindent consists of 1-morphism $(1 \, \Box \, h) \circ (g \, \Box \, 1) :x_0 \, \Box^\mathfrak{D} \, x_1 \to y_0 \, \Box^\mathfrak{D} \, y_1$ in $\mathfrak{D}$ together with an invertible modification

        \begin{equation}\label{eqn:DrinfeldCenterProduct1Morphism}
        \begin{tabular}{@{}cc@{}}
           $b^{gh}_{z}$ := & \stringdiagramfigure{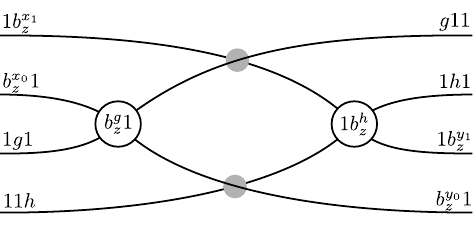}
        \end{tabular}
        \end{equation} given for object $z$ in $\mathfrak{C}$.

        \item [(7)] Interchangers of 1-morphisms are induced by those in $\mathfrak{D}$.
        
        \item [(8)] Monoidal product of 2-morphisms is just the monoidal product of the corresponding 2-morphisms in $\mathfrak{D}$.
    \end{itemize}
\end{definition}

\begin{lemma}\label{lem:drinfeld_centralizer_forgetful_2functor}
    By the above construction, there is a canonical forgetful 2-functor $\mathscr{Z}_1(F) \to \mathfrak{D}$, preserving the monoidal structures.
\end{lemma}

\begin{definition} \label{def:drinfeld_center}
    The \textit{Drinfeld center} of a monoidal 2-category $\mathfrak{C}$ is defined to be the monoidal centralizer $\mathscr{Z}_1(\mathrm{Id}_\mathfrak{C})$. We denote it by $\mathscr{Z}_1(\mathfrak{C})$.
\end{definition}

\begin{lemma} \label{lem:drinfeld_center_is_braided}
    Drinfeld center $\mathscr{Z}_1(\mathfrak{C})$ is equipped with a canonical braided monoidal structure.
\end{lemma}

\begin{proof}
    We define the braiding on objects $(x, b^x_{-}, R^x_{-,-})$ and $(y, b^y_{-}, R^y_{-,-})$ in $\mathscr{Z}_1(\mathfrak{C})$ as $b_{x,y} := b^x_{y}$. This braiding can be extended to an adjoint 2-natural equivalence. Additionally, for $(x, b^x_{-}, R^x_{-,-})$, $(y, b^y_{-}, R^y_{-,-})$ and $(z, b^z_{-}, R^z_{-,-})$ in $\mathscr{Z}_1(\mathfrak{C})$, we introduce invertible modifications $R_{x,y,z}:= R^x_{y,z}$ and $S^x_{y,z}:=1_{xyz}$. These assignments establish a braiding on Drinfeld center $\mathscr{Z}_1(\mathfrak{C})$ according to \cite[3.1]{Cr}.
\end{proof}

\begin{remark} \label{rmk:Z_1(C)=ΩEnd(∑C)}
    One can also realize the Drinfeld center 
    \begin{equation}
        \mathscr{Z}_1(\mathfrak{C}) \simeq \mathbf{Fun}_{\mathfrak{C} \boxtimes \mathfrak{C}^{1mp}}(\mathfrak{C},\mathfrak{C})
    \end{equation}
    by \cite[2.2.1]{D9}. For fusion 2-category $\mathfrak{C}$, the 3-category $\mathbf{Bimod}(\mathfrak{C},\mathfrak{C})$ of $(\mathfrak{C},\mathfrak{C})$-bimodule 2-categories is equipped with the relative 2-Deligne tensor product $\boxtimes_\mathfrak{C}$. $\mathfrak{C}$ is the monoidal unit in this monoidal 3-category. $\mathscr{Z}_1(\mathfrak{C}) \simeq \mathbf{Fun}_{\mathfrak{C} \boxtimes \mathfrak{C}^{1mp}}(\mathfrak{C},\mathfrak{C})$ is the endo-Hom 2-category on the monoidal unit, hence it is equipped with a canonical braided monoidal 2-category structure. Finally, these two braided monoidal 2-category structures on $\mathscr{Z}_1(\mathfrak{C})$ will agree with each other.\footnote{Depending on the exact definitions of Drinfeld center and relative 2-tensor product, these two braided monoidal 2-category structure might agree up to flipping the direction of the braidings.}
\end{remark}

\begin{remark} \label{rmk:central_action_of_drinfeld_center_on_morita_dual}
    The Drinfeld center is a Morita invariant: when $\mathfrak{C}$ is a monoidal 2-category, $\mathfrak{M}$ is a left $\mathfrak{C}$-module 2-category, we have an equivalence of braided monoidal 2-categories 
    \begin{equation}
        \mathscr{Z}_1(\mathfrak{C}) \simeq \mathscr{Z}_1(\mathbf{End}_\mathfrak{C}(\mathfrak{M})^{1mp}).
    \end{equation}
    Hence, it induces a monoidal 2-functor $\mathscr{Z}_1(\mathfrak{C}) \to \mathbf{End}_\mathfrak{C}(\mathfrak{M})^{1mp}$ which factors through the Drinfeld center, i.e. a central $\mathscr{Z}_1(\mathfrak{C})$-action on $\mathbf{End}_\mathfrak{C}(\mathfrak{M})^{1mp}$.

    In particular, when $\mathfrak{M} \simeq \Mod_\mathfrak{C}(A)$ for some algebra $A$ in $\mathfrak{C}$, assuming the relative tensor products over $A$ exist, we can identify $\mathbf{End}_\mathfrak{C}(\Mod_\mathfrak{C}(A))^{1mp}$ with $\Bimod_\mathfrak{C}(A,A)$, and get a central $\mathscr{Z}_1(\mathfrak{C})$-action on $\Bimod_\mathfrak{C}(A,A)$.

    More concretely, for an object $(x,b^x_{-},R^x_{-,-})$ in $\mathscr{Z}_1(\mathfrak{C})$, its image in $\Bimod_\mathfrak{C}(A,A)$ is $x \boxt A$, with the left $A$-action given by 
    \begin{equation}\label{eqn:left_action_on_image_of_half-braiding}
        l^M \colon A \boxt x \boxt A \xrightarrow{(b^x_{A})^{\bullet} \boxt 1_A} x \boxt A \boxt A \xrightarrow{1_x \boxt m} x \boxt A,
    \end{equation}
    where $(b^x)^\bullet$ is the 2-adjoint inverse of 2-natural isomorphism $b^x$, and the right $A$-action given by
    \begin{equation}\label{eqn:right_action_on_image_of_half-braiding}
        r^M \colon x \boxt A \boxt A \xrightarrow{1_x \boxt m} x \boxt A.
    \end{equation}
\end{remark}

\begin{remark} \label{rmk:braiding_as_section_of_drinfeld_center}
    A braiding on $\mathfrak{C}$ is equivalent to a monoidal section of the canonical forgetful 2-functor $\mathscr{Z}_1(\mathfrak{C}) \to \mathfrak{C}$, i.e. an embedding $\mathfrak{C} \to \mathscr{Z}_1(\mathfrak{C})$ preserving monoidal structures such that its composition with the forgetful 2-functor is equivalent to the identity 2-functor $\mathrm{Id}_\mathfrak{C}$. 
\end{remark}

\begin{notation}\label{not:braided_embedding_into_drinfeld_center}
    For a braided monoidal 2-category $\mathfrak{B}$, we denote the braided embedding of $\mathfrak{B}$ into its Drinfeld center $\mathscr{Z}_1(\mathfrak{B})$ by
    \begin{equation}
        \iota_+ \colon \mathfrak{B} \to \mathscr{Z}_1(\mathfrak{B}); \qquad x \mapsto (x, b_{x,-}, R_{x,-,-}).
    \end{equation}
    Similarly, we can replace the braiding by its reverse to get another braided embedding
    \begin{equation}
        \iota_- \colon \mathfrak{B}^{2mp} \to \mathscr{Z}_1(\mathfrak{B}); \qquad x \mapsto (x, b^\bullet_{-,x}, R^\bullet_{-,-,x}),
    \end{equation}
    where $b^\bullet$ is the 2-adjoint inverse of 2-natural isomorphism $b$, and $R^\bullet$ is the induced 2-adjoint inverse of $R$.
\end{notation}

\subsection{Full center}
Let $\mathfrak{C}$ be a monoidal 2-category, and $\mathbf{U} \colon \Z_1(\mathfrak{C})\to \mathfrak{C}$ be the forgetful 2-functor. Let $A$ be an algebra in $\mathfrak{C}$. The 2-category of \emph{commuting half-braidings} of $A$ is defined as follows.

\begin{definition}\label{def:commuting_half-braidings}
    A \emph{commuting half-braiding} of $A$ consists of:
    \begin{enumerate}[nosep,label=(\roman*)] 
        \item A half-braiding $(x,b^x,R^x)$ in $\mathfrak{C}$, or equivalently an object in $\Z_1(\mathfrak{C})$;
        
        \item A 1-morphism $f \colon x \to A$ in $\mathfrak{C}$;
        
        \item A 2-isomorphism
        \begin{equation}
            \psi^f \colon m \circ (f \, \Box \, 1_A) \circ b^x_{A} \Rightarrow m \circ (1_A \, \Box \, f),
        \end{equation} satisfying the same coherence equations \eqref{eqn:commuting_object_associativity} and \eqref{eqn:commuting_object_unitality} as in Definition \ref{def:commuting_objects}.
    \end{enumerate}
\end{definition}

\begin{definition}\label{def:commuting_half-braidings_morphisms}
    A 1-morphism of commuting half-braidings from $(x,b^x,R^x,f,\psi^f)$ to $(y,b^y,R^y,g,\psi^g)$ consists of:
    \begin{enumerate}[nosep,label=(\roman*)]
        \item A 1-morphism $(h,b^h)$ in $\Z_1(\mathfrak{C})$ from $(x,b^x,R^x)$ to $(y,b^y,R^y)$;
        
        \item A 2-isomorphism $\chi^h \colon f \to g \circ h$ in $\mathfrak{C}$, satisfying the coherence equation \eqref{eqn:commuting_object_1morphism} as in Definition \ref{def:commuting_objects_morphisms}.
    \end{enumerate}

    A 2-morphism of commuting half-braidings from $(h,b^h,\chi^h)$ to $(k,b^k,\chi^k)$, which are 1-morphisms from $(x,b^x,R^x,f,\psi^f)$ to $(y,b^y,R^y,g,\psi^g)$, consists of a 2-morphism $\xi \colon h \to k$ in $\Z_1(\mathfrak{C})$ subject to the coherence equation \eqref{eqn:commuting_object_2morphism} as in Definition \ref{def:commuting_objects_morphisms}.
\end{definition}

\begin{definition}\label{def:full_center}
The \emph{full center} of an algebra $A$ is the terminal object in the 2-category of commuting half-braidings of $A$, if it exists. We denote it by $\mathbf{Z}_1(A)$.
\end{definition}

Suppose the forgetful 2-functor $\mathbf{U} \colon \Z_1(\mathfrak{C})\to \mathfrak{C}$ from Lemma \ref{lem:drinfeld_centralizer_forgetful_2functor} has a right adjoint $\mathbf{U}^R \colon \mathfrak{C} \to \Z_1(\mathfrak{C})$. By Construction \ref{cstr:lax_monoidal_structure_on_the_right_adjoint}, $\mathbf{U}^R$ is equipped with a canonical lax monoidal structure. For an algebra $A$ in $\mathfrak{C}$, $\mathbf{U}^R(A)$ is an algebra in $\Z_1(\mathfrak{C})$.

Meanwhile, by Remark \ref{rmk:central_action_of_drinfeld_center_on_morita_dual}, we have a central $\Z_1(\mathfrak{C})$-action on $\Bimod_\mathfrak{C}(A,A)$, which induces a forgetful 2-functor $\mathbf{V} \colon \Z_1(\mathfrak{C}) \to \Bimod_\mathfrak{C}(A,A)$. Suppose $\mathbf{V}$ has a right adjoint $\mathbf{V}^R \colon \Bimod_\mathfrak{C}(A,A) \to \Z_1(\mathfrak{C})$. By Construction \ref{cstr:lax_monoidal_structure_on_the_right_adjoint}, $\mathbf{V}^R$ is equipped with a canonical lax monoidal structure. For the algebra $A$ in $\Bimod_\mathfrak{C}(A,A)$, $\mathbf{V}^R(A)$ is also an algebra in $\Z_1(\mathfrak{C})$.

\begin{theorem}\label{prop:full-center}
For any algebra $A$ in $\mathfrak{C}$, its full center has two equivalent characterizations:
\begin{enumerate}[nosep,label=(\roman*)]
    \item $\mathbf{Z}_1(A) \simeq \mathbf{V}^R(A)$,
    
    \item When the component of counit $\epsilon_A \colon \mathbf{U} \mathbf{U}^R(A) \to A$ is epic, $\mathbf{Z}_1(A) \simeq \Zl(\mathbf{U}^R(A))$.
\end{enumerate}
\end{theorem}

\begin{proof}
For any half-braiding $X = (x,b^x,R^x)$ in $\mathfrak{C}$, the 2-adjunction between $\mathbf{V}$ and $\mathbf{V}^R$ gives a natural equivalence of Hom categories
\begin{equation}
    \mathbf{Hom}_{\Bimod_\mathfrak{C}(A,A)}(\mathbf{V}(X), A) \simeq \mathbf{Hom}_{\Z_1(\mathfrak{C})}(X,\mathbf{V}^R(A)).
\end{equation}
By Remark \ref{rmk:central_action_of_drinfeld_center_on_morita_dual}, $\mathbf{V}(X) \simeq x \boxt A$ with the left $A$-action given by \eqref{eqn:left_action_on_image_of_half-braiding}, and the right $A$-action given by \eqref{eqn:right_action_on_image_of_half-braiding}. We can verify that the left-hand side is equivalent to the category of data promoting the half-braiding $X$ to a commuting half-braiding of $A$:
\begin{itemize}
    \item For a commuting half-braiding $(X,f,\psi^f)$, consider the 1-morphism $x \boxt A \xrightarrow{f \boxt 1_A} A \boxt A \xrightarrow{m} A$ with the left $A$-module structure given by
    \begin{equation}
        \begin{array}{@{}c@{\quad}c@{}}
        \stringdiagramfigure{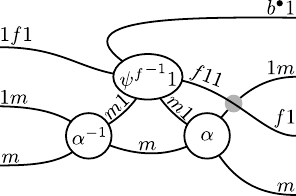} & ,
        \end{array}
    \end{equation}
    and the right $A$-module structure given by
    \begin{equation}
        \begin{array}{@{}c@{\quad}c@{}}
        \stringdiagramfigure{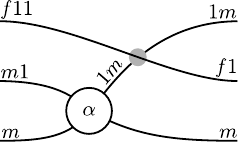} & .
        \end{array}
    \end{equation}

    \item Conversely, for an $(A,A)$-module 1-morphism $g \colon x \boxt A \to A$, we can construct a 1-morphism $x \xrightarrow{1_x \boxt i} x \boxt A \xrightarrow{g} A$ with a 2-isomorphism $\psi^g$ given by
    \begin{equation}
        \begin{array}{@{}c@{\quad}c@{}}
        \stringdiagramfigure{figures/string-diagrams/full_center/half-braiding_from_bimodule_morphism.pdf} & ,
        \end{array}
    \end{equation}
    where $\psi^l$ is the left $A$-module structure of $g$, and $\psi^r$ is the right $A$-module structure of $g$.
\end{itemize}

The 2-functor
\begin{equation}
    \mathbf{Hom}_{\Bimod_\mathfrak{C}(A,A)}(\mathbf{V}(-), A) \colon \Z_1(\mathfrak{C})^{1op} \to \mathbf{Cat}
\end{equation}
is representable by $\mathbf{V}^R(A)$; meanwhile, this 2-functor is representable if and only if the full center $\mathbf{Z}_1(A)$ exists as the terminal object in the Grothendieck construction, i.e., the unstraightened 2-category of commuting half-braidings of $A$. Hence, we have $\mathbf{Z}_1(A) \simeq \mathbf{V}^R(A)$.

To prove the second statement, consider the 2-adjunction between $\mathbf{U}$ and $\mathbf{U}^R$:
\begin{equation}
    \mathbf{Hom}_\mathfrak{C}(\mathbf{U}(X), A) \simeq \mathbf{Hom}_{\Z_1(\mathfrak{C})}(X,\mathbf{U}^R(A)).
\end{equation}
Let us denote the unit by $\eta \colon \mathrm{Id}_{\Z_1(\mathfrak{C})} \to \mathbf{U}^R \mathbf{U}$ and the counit by $\epsilon \colon \mathbf{U} \mathbf{U}^R \to \mathrm{Id}_\mathfrak{C}$. We would like to compare commuting half-braidings of $A$ with left commuting objects over $\mathbf{U}^R(A)$ internal to $\Z_1(\mathfrak{C})$. Notice that the underlying 1-morphisms are exactly mates under this 2-adjunction. So we only need to compare the 2-isomorphisms associated.

Given a commuting half-braiding $(X,f,\psi^f)$, we can always put a left commuting object structure on its mate:
\begin{itemize}
    \item The underlying 1-morphism is induced by
    \begin{equation}
        X \xrightarrow{\eta_{X}} \mathbf{U}^R \mathbf{U}(X) \xrightarrow{\mathbf{U}^R(f)} \mathbf{U}^R(A),
    \end{equation}
    
    \item The 2-isomorphism is given by
    \begin{equation}
        \begin{array}{@{}c@{\quad}c@{}}
        \stringdiagramfigure{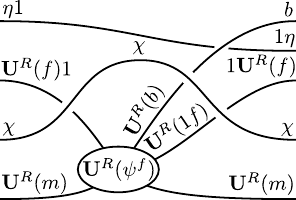} & ,
        \end{array}
    \end{equation}
    where $\chi$ is the lax monoidal structure of $\mathbf{U}^R$ constructed in Construction \ref{cstr:lax_monoidal_structure_on_the_right_adjoint}.
\end{itemize}

In particular, the universal property induces a braided algebra 1-morphism $\mathbf{Z}_1(A) \to \Zl(\mathbf{U}^R(A))$ in $\Z_1(\mathfrak{C})$.

Conversely, given a left commuting object $(X,g,\psi^g)$ over $\mathbf{U}^R(A)$, consider its mate under the 2-adjunction:
\begin{itemize}
    \item The underlying 1-morphism is induced by
    \begin{equation}
        \mathbf{U}(X) \xrightarrow{\mathbf{U}(g)} \mathbf{U} \mathbf{U}^R(A) \xrightarrow{\epsilon_A} A,
    \end{equation}

    \item However, the 2-isomorphism $\psi^g$ is directly transferred to 
    \begin{equation}
        \begin{array}{@{}c@{\quad}c@{}}
        \stringdiagramfigure{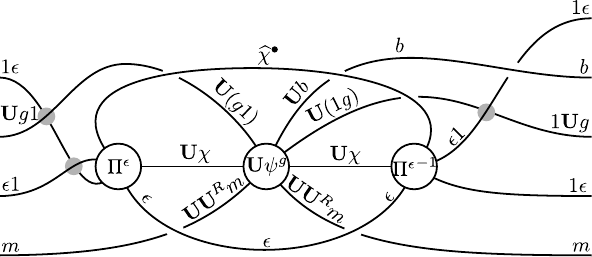} & ,
        \end{array}
    \end{equation}
    where $\widehat{\chi}^\bullet$ is the adjoint inverse of the oplax monoidal structure of $\mathbf{U}$, $\chi$ is the lax monoidal structure of $\mathbf{U}^R$ constructed by \eqref{eqn:lax_monoidal_structure_on_right_adjoint}, $\Pi^\epsilon$ is the compatibility data between monoidal structures and the counit constructed by \eqref{eqn:counit_monoidal_structure}, see Construction \ref{cstr:lax_monoidal_structure_on_the_right_adjoint}. To construct the 2-isomorphism of the commuting half-braiding, we need to descend through the 1-morphism $1_{\mathbf{U}(X)} \boxt \epsilon_A$. This is exactly the place where we need the assumption that $\epsilon_A$ is epic, so that we can use the right cancellation property here.
\end{itemize}
This induces another braided algebra 1-morphism $\Zl(\mathbf{U}^R(A)) \to \mathbf{Z}_1(A)$ in $\Z_1(\mathfrak{C})$. It can be verified that these two braided algebra 1-morphisms are mutually inverse, hence we have $\mathbf{Z}_1(A) \simeq \Zl(\mathbf{U}^R(A))$.
\end{proof}

\begin{corollary}\label{cor:full_center_is_morita_invariant}
The full center of an algebra $A$ in $\mathfrak{C}$ is Morita invariant: for any algebra $B$ in $\mathfrak{C}$ such that $\Mod_\mathfrak{C}(A) \simeq \Mod_\mathfrak{C}(B)$ as left $\mathfrak{C}$-module 2-categories, we have $\mathbf{Z}_1(A) \simeq \mathbf{Z}_1(B)$ as braided algebras in $\Z_1(\mathfrak{C})$.
\end{corollary}

\begin{remark}
    Similarly, when $\epsilon_A$ is epic, we have $\mathbf{Z}_1(A) \simeq \Zr(\mathbf{U}^R(A))$.
\end{remark}

\section{Examples and Applications}

\subsection{Drinfeld center of monoidal categories}
When $\mathfrak{B} = \mathbf{2Vect}$, algebras in $\mathfrak{B}$ are finite semisimple monoidal categories, and modules in $\mathfrak{B}$ are finite semisimple module categories. The left, right, and full centers of monoidal categories coincide, all equivalent to the Drinfeld center of monoidal categories \cite{Drinfeld,Majid,JS91}. The Morita invariance of the Drinfeld center is a well-known result in the theory of fusion categories, see \cite{ENO05}.

\subsection{Crossed braided center of graded monoidal categories}
Next, let us consider the case $\mathfrak{C} = \mathbf{2Vect}_G$ is the 2-category of $G$-graded finite semisimple categories, for a finite group $G$.

\begin{definition}\label{def:graded_monoidal_categories}
    A finite semisimple $G$-\emph{graded monoidal category} consists of:
    \begin{itemize}[noitemsep]
        \item A finite semisimple monoidal category $\mathcal{A}$, together with a decomposition $\mathcal{A} = \bigoplus_{g \in G} \mathcal{A}_g$, such that the monoidal product $\otimes$ restricts to a functor $\mathcal{A}_g \boxtimes \mathcal{A}_h \to \mathcal{A}_{gh}$ for all $g,h \in G$;
        
        \item A unit object $\mathbf{I} \in \mathcal{A}_e$, where $e \in G$ is the identity element.
    \end{itemize}
\end{definition}

\begin{definition}\label{def:graded_module_categories}
    A finite semisimple $G$-\emph{graded left module category} over a $G$-graded monoidal category $\mathcal{A}$ consists of:
    \begin{itemize}[nosep]
        \item A finite semisimple category $\mathcal{M}$, together with a decomposition $\mathcal{M} = \bigoplus_{g \in G} \mathcal{M}_g$;
        
        \item A left $\mathcal{A}$-action $\triangleright: \mathcal{A} \boxtimes \mathcal{M} \to \mathcal{M}$, which restricts to a functor $\mathcal{A}_g \boxtimes \mathcal{M}_h \to \mathcal{M}_{gh}$ for all $g,h \in G$.
    \end{itemize}
\end{definition}

\begin{lemma}
    Algebras in $\mathbf{2Vect}_G$ are finite semisimple $G$-graded monoidal categories, and modules in $\mathbf{2Vect}_G$ are finite semisimple $G$-graded module categories.
\end{lemma}

\begin{definition}\label{def:crossed_braided_categories}
    A finite semisimple $G$-\emph{crossed braided monoidal category} consists of:
    \begin{itemize}[nosep]
        \item A finite semisimple $G$-graded monoidal category $\mathcal{A} = \bigoplus_{g \in G} \mathcal{A}_g$;
        
        \item A $G$-action on $\mathcal{A}$ by monoidal autoequivalences, such that $g \triangleright \mathcal{A}_h \subseteq \mathcal{A}_{ghg^{-1}}$ for all $g,h \in G$;
        
        \item A natural isomorphism, called the \emph{crossed braiding}
        \begin{equation}
            \beta_{X,Y}: X \otimes Y \xrightarrow{\sim} g(Y) \otimes X,
        \end{equation}
        for all $X \in \mathcal{A}_g$, $Y \in \mathcal{A}$, satisfying additional coherence conditions (see \cite{GNN09,TV13} for details).
    \end{itemize}
\end{definition}

\begin{definition}\label{def:crossed_braided_center}
    The \emph{crossed braided center} of a $G$-graded monoidal category $\mathcal{A}$ is defined to be
    \begin{equation}
        \mathcal{Z}_G(\mathcal{A}) := \mathbf{Fun}_{\mathcal{A}_e \boxtimes \mathcal{A}^{\mathrm{mop}}}(\mathcal{A}, \mathcal{A}).
    \end{equation}

    More explicitly, 
    \begin{enumerate}[nosep,label=(\arabic*)]
        \item An object is a pair $(X, \beta_{X,-})$, where $X \in \mathcal{A}$ and $\beta_{X,-}$ is a natural isomorphism satisfying half of the crossed braiding conditions.
        
        \item A morphism between two objects $(X, \beta_{X,-})$ and $(Y, \beta_{Y,-})$ is a morphism $f: X \to Y$ in $\mathcal{A}$ intertwining the half-braidings.
    \end{enumerate}
\end{definition}

\begin{proposition}\label{prop:crossed_braided_center_is_full_center}
    The crossed braided center $\mathcal{Z}_G(\mathcal{A})$ is a $G$-crossed braided monoidal category, and it is equivalent to the full center of $\mathcal{A}$ as an algebra in $\mathbf{2Vect}_G$.
\end{proposition}

\begin{remark}
    The full centers of $G$-graded fusion categories have canonical $G$-crossed braided fusion structure as introduced in \cite{GNN09}; they are also generalized to general monoidal cases in \cite{TV13}. The Morita invariance of the $G$-crossed braided center appears in \cite{ENO10,GJS21}.
\end{remark}

\begin{remark}
    By contrast, the left and right centers of a $G$-crossed fusion category as algebras in $\Z_1(\mathbf{2Vect}_G)$ differ from the full center: unlike in the ungraded case of $\mathbf{2Vect}$, where the left, right, and full centers coincide, the grading breaks this coincidence. To the best of our knowledge, these one-sided centers of crossed monoidal categories have not been described explicitly in the literature. They are nonetheless quite accessible: Theorem \ref{prop:full-center} expresses the full center in terms of the left (or right) center of the right adjoint of the forgetful 2-functor. We leave their explicit structure in the graded setting as a natural direction for the interested reader.
\end{remark}

\begin{remark}
More generally, one can also twist the 2-associator of $\mathbf{2Vect}_G$ by a 4-cocycle $\pi \in \operatorname{Z}^4(\mathrm{B}G;\mathbb{C}^\times)$. Algebras in $\mathfrak{C} = \mathbf{2Vect}_G^\pi$ are finite semisimple $\pi$-twisted $G$-graded monoidal categories, and modules in $\mathfrak{C}$ are finite semisimple $\pi$-twisted $G$-graded module categories. The Drinfeld center $\Z_1(\mathbf{2Vect}_G^\pi)$ has been studied thoroughly in \cite{KTZ20}. The full centers of $\pi$-twisted $G$-graded fusion categories have canonical $\pi$-twisted $G$-crossed braided fusion structure. See the author's previous work \cite{Xu} for a detailed list of coherence data.
\end{remark}

\subsection{Centers of central module monoidal categories}
Let $\mathcal{B}$ be a braided fusion 1-category, and $\mathfrak{C} = \mathbf{Mod}(\mathcal{B})$ be the 2-category of finite semisimple right $\mathcal{B}$-module categories. 

\begin{definition}\label{def:central_module_category}
    A \emph{central $\mathcal{B}$-module monoidal category} consists of a monoidal category $\mathcal{C}$, together with a braided monoidal functor into its Drinfeld center: $\mathcal{B} \to \mathcal{Z}_1(\mathcal{C})$.
\end{definition}

\begin{definition}\label{def:left_module_category_over_central_module_category}
    A \emph{left module category over a central $\mathcal{B}$-module monoidal category} $\mathcal{C}$ is a left $\mathcal{C}$-module category $\mathcal{M}$, together with a compatible right $\mathcal{B}$-module structure on $\mathcal{M}$.
\end{definition}

\begin{lemma}
    Algebras in $\mathbf{Mod}(\mathcal{B})$ are finite semisimple central $\mathcal{B}$-module monoidal categories, and modules in $\mathbf{Mod}(\mathcal{B})$ are finite semisimple module categories over such algebras.
\end{lemma}

\begin{remark}
    The identification of a central $\mathcal{B}$-module structure on a monoidal category $\mathcal{A}$ with a braided functor $\mathcal{B} \to \mathcal{Z}_1(\mathcal{A})$ into the Drinfeld center goes back to the notion of a central functor \cite{DGNO}, and the resulting theory of central $\mathcal{B}$-module categories was developed by Davydov--Nikshych \cite{DN13,DN21}; see also \cite{Xu}.
\end{remark}

Recall that the Drinfeld center $\mathcal{Z}_1(\mathbf{Mod}(\mathcal{B}))$ is equivalent to the 2-category of braided $\mathcal{B}$-module categories \cite{DN21}. Braided algebras there are precisely braided monoidal categories $\mathcal{M}$ equipped with a braided functor $\mathcal{B} \to \mathcal{M}$, see also \cite{JFR}.

\begin{proposition}\label{prop:full_center_of_central_module_category}
    The full center of a central $\mathcal{B}$-module monoidal category $\mathcal{C}$ is its ordinary Drinfeld center $\mathcal{Z}_1(\mathcal{C})$, together with the same braided functor $\mathcal{B} \to \mathcal{Z}_1(\mathcal{C})$.
\end{proposition}

\begin{remark}
    More generally, if $A$ is a braided algebra in a braided fusion 2-category $\mathfrak{B}$, Proposition \ref{prop:algebras-in-modules} provides a useful base-change principle: algebras in the module 2-category $\Mod_{\mathfrak{B}}(A)$ can be described as algebras $B$ in $\mathfrak{B}$ equipped with a braided algebra map $A\to \Zl(B)$. From another perspective, this is also a braided algebra in $\Z_1(\mathfrak{B})$, which is the full center of $A$ as an algebra in $\mathfrak{B}$.
\end{remark}

Finally, let us comment on how (de-)equivariantization \cite{DGNO} relates full centers in both pictures. Denote the 2-category of finite semisimple categories with $G$-action by
\begin{equation}
    \mathbf{2Rep}(G) := \mathbf{Fun}(\mathrm{B}G, \mathbf{2Vect}),
\end{equation} and the 2-category of finite semisimple $\mathbf{Rep}(G)$-module categories by
\begin{equation}
    \Sigma \mathbf{Rep}(G) := \mathbf{Mod}(\mathbf{Rep}(G)).
\end{equation}
Recall that there is an equivalence of symmetric fusion 2-categories:
\begin{equation}
    \mathbf{2Rep}(G) \simeq \Sigma \mathbf{Rep}(G),
\end{equation}
witnessed by relative tensor product over the invertible bimodule category $\mathbf{Vect}$.

Hence, (de-)equivariantization induces a braided equivalence of Drinfeld centers:
\begin{equation}
    \Z_1(\mathbf{2Rep}(G)) \simeq \Z_1(\Sigma \mathbf{Rep}(G)).
\end{equation}
Consider the monoidal unit: in $\mathbf{2Rep}(G)$, it is $\mathbf{Vect}$ with the trivial $G$-action; in $\Sigma \mathbf{Rep}(G)$, it is $\mathbf{Rep}(G)$ with the regular action on itself. The full center of the monoidal unit in $\mathbf{2Rep}(G)$, by Proposition \ref{prop:crossed_braided_center_is_full_center}, is equivalent to $\mathbf{Vect}_G$ with the conjugation $G$-crossed braided structure. The full center of the monoidal unit in $\Sigma \mathbf{Rep}(G)$, by Proposition \ref{prop:full_center_of_central_module_category}, is equivalent to $\mathcal{Z}_1(\mathbf{Rep}(G))$. Indeed, one can check that $\mathcal{Z}_1(\mathbf{Rep}(G))$ is the equivariantization of $\mathbf{Vect}_G$ with the conjugation $G$-crossed braided structure, and the two full centers are equivalent under the braided equivalence $\Z_1(\mathbf{2Rep}(G)) \simeq \Z_1(\Sigma \mathbf{Rep}(G))$.

\bibliographystyle{alpha}
\bibliography{references}

\end{document}